\documentclass[11pt]{article}

\usepackage[margin=1in]{geometry}
\usepackage{amsmath,amssymb,amsthm,mathtools}
\usepackage{xcolor}
\usepackage{graphicx,tikz}
\usepackage[colorlinks=true,linkcolor=blue,citecolor=red,urlcolor=blue]{hyperref}

\newtheorem{theorem}{Theorem}
\newtheorem{proposition}{Proposition}
\newtheorem{lemma}{Lemma}

\newcommand{\Bern}{\operatorname{Bern}}

\newcommand{\E}{\mathbb E}
\newcommand{\Pbb}{\mathbb P}
\newcommand{\Expect}{\mathbb{E}}
\newcommand{\expect}[1]{\mathbb{E}\left[#1\right]}
\newcommand{\Prob}{\mathbb{P}}

\newcommand{\iprod}[2]{\left \langle #1, #2 \right\rangle}
\newcommand{\Iprod}[2]{\langle #1, #2 \rangle}
\newcommand{\indc}[1]{{\mathbf{1}\left\{{#1}\right\}}}

\newcommand{\KL}{\operatorname{KL}}

\newcommand{\Post}{\operatorname{Post}}
\newcommand{\one}{\mathbf 1}
\newcommand{\cL}{\mathcal L}

\newcommand{\cR}{\mathcal R}
\newcommand{\ip}[2]{\langle #1,#2\rangle}

\newcommand{\bsell}{\boldsymbol{\ell}}
\newcommand{\bfa}{\mathbf{a}}

\newcommand{\sfA}{\mathsf{A}}

\newcommand{\sfT}{\mathsf{T}}

\newcommand{\bY}{\mathbf{Y}}
\newcommand{\bZ}{\mathbf{Z}}
\newcommand{\bG}{\mathbf{G}}

\newcommand{\calA}{\mathcal{A}}

\newcommand{\calH}{\mathcal{H}}

\newcommand{\calK}{\mathcal{K}}

\newcommand{\naturals}{\mathbb{N}}

\newcommand{\stepa}[1]{\overset{\rm (a)}{#1}}
\newcommand{\stepb}[1]{\overset{\rm (b)}{#1}}
\newcommand{\stepc}[1]{\overset{\rm (c)}{#1}}
\newcommand{\stepd}[1]{\overset{\rm (d)}{#1}}

\newcommand{\Law}{\cL}

\usepackage{prettyref}
\newrefformat{eq}{(\ref{#1})}
\newrefformat{e}{(\ref{#1})}
\newrefformat{thm}{Theorem~\ref{#1}}
\newrefformat{th}{Theorem~\ref{#1}}
\newrefformat{chap}{Chapter~\ref{#1}}
\newrefformat{sec}{Section~\ref{#1}}
\newrefformat{seca}{Section~\ref{#1}}
\newrefformat{algo}{Algorithm~\ref{#1}}
\newrefformat{fig}{Fig.~\ref{#1}}
\newrefformat{tab}{Table~\ref{#1}}
\newrefformat{rmk}{Remark~\ref{#1}}
\newrefformat{clm}{Claim~\ref{#1}}
\newrefformat{def}{Definition~\ref{#1}}
\newrefformat{cor}{Corollary~\ref{#1}}
\newrefformat{lmm}{Lemma~\ref{#1}}
\newrefformat{lem}{Lemma~\ref{#1}}
\newrefformat{prop}{Proposition~\ref{#1}}
\newrefformat{pr}{Proposition~\ref{#1}}
\newrefformat{app}{Appendix~\ref{#1}}
\newrefformat{apx}{Appendix~\ref{#1}}
\newrefformat{ex}{Example~\ref{#1}}
\newrefformat{exer}{Exercise~\ref{#1}}
\newrefformat{soln}{Solution~\ref{#1}}

\newcommand{\ER}{Erd\"{o}s-R\'enyi}

\title{Sharp detection thresholds for random geometric graphs \\ with polynomial average degrees}
\author{Hang Du, Cheng Mao, Nike Sun, Yihong Wu, and Jiaming Xu\thanks{
H.\ Du and N.\ Sun are with the Department of Mathematics, Massachusetts Institute of Technology, Cambridge, Massachusetts, USA, \texttt{\{hangdu,nsun\}@mit.edu}.
C.\ Mao is with the School of Mathematics, Georgia Institute of Technology, Atlanta, Georgia, USA, \texttt{cheng.mao@math.gatech.edu}.
Y.\ Wu is with the Department of Statistics and Data Science, Yale University, New Haven, Connecticut, USA, \texttt{yihong.wu@yale.edu}.
J.\ Xu is with The Fuqua School of Business, Duke University, Durham, North Carolina, USA, \texttt{jx77@duke.edu}.
}}
\date{}

\begin{document}
\maketitle

\begin{abstract}
We study the random geometric graph formed by connecting $n$ independent uniform points on $\mathbb S^{d-1}$ whenever their inner product exceeds a threshold chosen to give edge density $p$. We establish the conjectured dimensionality threshold for indistinguishability from the Erd\"{o}s-R\'enyi graph $G(n,p)$ throughout the sparse regime with polynomially growing average degree. Specifically, for every fixed $\delta>0$, if $p\to0$, $p\ge n^{-1+\delta}$, and $d\gg (np\log(1/p))^3$, then the Kullback-Leibler (KL) divergence between the two distributions vanishes.

Our proof builds on the posterior analysis framework introduced in previous work of the authors, combining a truncated second moment analysis with Fourier and spherical harmonic expansions. As in prior works, we reduce the KL bound to controlling a truncated $\chi^2$-divergence $\Delta$ for the neighborhood of the last vertex conditional on the graph induced by the preceding vertices. 
We then bound $\Delta$ using a Fourier expansion of the conditional likelihood ratio for this neighborhood, where the high-frequency contribution is made negligible by truncation to a typical-degree event, while a spherical harmonic expansion approximates the low-frequency contribution in terms of posterior moments of bounded-degree harmonic polynomials. The key observation is that these moments and $\Delta$ satisfy a coupled system of cavity recurrences, yielding a self-bounding inequality that implies the desired KL bound.

%

\end{abstract}

\tableofcontents

\section{Introduction}

High-dimensional random geometric graphs (RGGs) have attracted substantial
attention in recent years across probability, statistics, and
theoretical computer science. An extensive and rapidly growing literature has
investigated the classical hard-threshold
model~\cite{devroye2011high,bubeck2016testing,brennan2020phase,liu2022testing,bangachev2024fourier,bangachev2025sandwiching},
soft and noisy models with smooth connection
kernels~\cite{liu2023phase,liu2023probabilistic,mao2026random}, bipartite models~\cite{brennan2021finetti,gobel2026information}, and variants based
on different latent geometries or algebraic
structures~\cite{bangachev2024detection,bangachev2025random}; see also the
survey~\cite{duchemin2023random}. Despite significant activity and substantial
progress, a simple and fundamental question has remained largely unresolved:
at what dimension does a high-dimensional RGG become statistically
indistinguishable from an Erd\H{o}s--R\'enyi random graph?

To formulate the problem, let $x_1,\ldots,x_n$ be independent and uniform on
the unit sphere $\mathbb S^{d-1}\subset\mathbb R^d$.  The hard-threshold
spherical RGG connects distinct vertices $i,j$ when
$\ip{x_i}{x_j}\ge\tau$, where $\tau=\tau(d,p)$ is chosen so that
\[
 \Pbb\{\ip{x_1}{x_2}\ge\tau\}=p.
\]
We write $\Pbb_{\mathrm{RGG}}$ for the law of this graph and
$\Pbb_{\mathrm{ER}}$ for the law of the Erd\H{o}s--R\'enyi graph $G(n,p)$.
The two models have identical single-edge marginals; the testing problem asks
whether the dependence created by the latent geometry can nevertheless be
detected from the graph alone.

This  hypothesis testing problem was introduced by Bubeck, Ding,
Eldan, and R\'acz~\cite{bubeck2016testing}. For constant $p$,
they established the sharp transition $d\asymp n^3$: the signed triangle
statistic detects the geometry when $d\ll n^3$, whereas the two models are
indistinguishable when $d\gg n^3$. Liu, Mohanty,
Schramm, and Yang~\cite{liu2022testing} subsequently showed that, when
$p=o(1)$, signed triangles achieve detection whenever
$
 d\ll (np\log(1/p))^3
$
and further conjectured that the two
models should be indistinguishable when
$d\gg (np\log(1/p))^3$.
Despite significant progress over many
years~\cite{brennan2020phase,liu2022testing}, the best-known general
impossibility condition was $d\gg n^3p^2\operatorname{polylog}(n)$, which differs from the
conjectured threshold by a polynomial factor throughout most of the sparse
regime. The closest result to the conjectured scaling was obtained in the
constant-average-degree regime, where indistinguishability was proved under
$d\gg(\log n)^{36}$~\cite{liu2022testing}, leaving only a polylogarithmic gap.

Progress toward the conjecture accelerated in 2026. The recent
advance~\cite{du2026resolution} shows that detection is impossible when
$d\gg (np\log(1/p))^3$ and $d\ge(1+\epsilon)n$ for a fixed $\epsilon>0$,
thereby settling the conjecture whenever $p\gtrsim n^{-2/3}/\log n$.
Independently, Hunter,
Milojevi\'c, and Sudakov~\cite{hunter2026distinguishability} obtained the
conjectured threshold in the range
$n^{-1/5}(\log n)^C\le p\le 1/3$, for a sufficiently large constant $C$, by a
different pointwise likelihood-comparison method.

The present paper resolves the conjecture throughout the sparse regime with
polynomially growing average degree.
The main result shows that when the dimension is above the conjectured threshold, 
the RGG and the \ER~graph are statistically indistinguishable with vanishing
Kullback-Leibler (KL) divergence.

\begin{theorem}
\label{thm:main-detection}
Fix any constant $\delta>0$.  Suppose that
\begin{equation}
 p\to0,\qquad p\ge n^{-1+\delta},\qquad
 d\gg (np\log(1/p))^3.
 \label{eq:regime}
\end{equation}
Then
\[
 \KL(\Pbb_{\mathrm{RGG}}\Vert\Pbb_{\mathrm{ER}})=o(1).
\]
\end{theorem}

\paragraph{Proof strategy.}
Our proof builds on the posterior-analysis framework introduced for soft
RGGs in~\cite{mao2026random}. A central insight of that work is that proving a
sharp detection threshold requires understanding the posterior distribution
of the latent positions given the observed graph. In particular, 
one needs to show that,  to
estimate $\langle x_1,x_2\rangle$, the information contained in
the entire graph beyond the single edge between vertices $1$ and $2$ is negligible. 
This perspective was subsequently used in~\cite{du2026resolution}
to establish the conjectured threshold for hard RGGs in the regime
$d\ge(1+\epsilon)n$. The
posterior bounds there rely on an information-theoretic comparison argument
and strong log-concavity, an approach that fundamentally requires $d>n$.

The main innovation of the present work is a different approach to posterior
analysis that remains effective when $d<n$. We leverage spherical-harmonic
expansions to derive a cavity recurrence that relates the posterior moment of
one harmonic polynomial to those of other harmonic polynomials with controlled
multidegrees.
Harmonic selection rules restrict the multidegrees that can appear in the
recurrence, while a degree-dependent normalization makes the recurrence
contractive, yielding the desired posterior bounds. 

To translate these posterior bounds into indistinguishability,
following~\cite{liu2022testing,mao2026random,du2026resolution}, we reveal the graph one vertex at a time and
use the KL chain rule to reduce the problem to controlling the conditional KL of neighborhoods
at each step by Fourier expansion, with Fourier coefficients given by certain posterior moments. As observed in \cite[Section~4]{du2026resolution}, the na\"ive bound
using unconditional moments fails when $d<n$. We overcome this
obstruction by restricting to the event that the newly revealed vertex has
typical degree. This truncation  makes the high-frequency Fourier remainder  negligible,
while the remaining low-frequency  terms are controlled by
the posterior bounds obtained from the cavity recurrence. Combining these
ingredients gives the desired vanishing KL divergence.

\paragraph{Organization.}
Section~\ref{sec:pf-main-thm} reduces the theorem to posterior bounds for 
spherical harmonics of bounded total degrees and shows how these bounds complete the proof of~\prettyref{thm:main-detection}.  Sections
\ref{sec:cavity-recurrence} and~\ref{sec:closing-recursion} derive and close the
recurrence for these posterior moments. 
Sections~\ref{sec:high-level-fourier} and~\ref{sec:high-harmonic-remainder} prove, respectively, the high-frequency
Fourier remainder and high-degree harmonic remainder estimates used earlier.

\paragraph{Notation.}
For a positive integer $m$, write $[m]=\{1,\ldots,m\}$.  Let
$G_m$ denote the random geometric graph induced by the first $m$ latent
points; in particular, $G_n$ is the observed graph.  Write
$W_{ij}=\langle x_i,x_j\rangle$.  Throughout the proof, we abbreviate
\begin{equation}\label{e:def.L.eps}
 L \equiv \log(1/p),\qquad
 \varepsilon_n
 \equiv \frac{(npL)^3}{d}=o(1).
\end{equation}
Constants may depend on the fixed parameter $\delta$ and on fixed truncation
cutoffs, but not on $n,p$, or $d$.

Define the centered standardized cap kernel
\begin{equation}\label{e:def.kappa}
 \kappa(u)=\frac{\one\{u\ge\tau\}-p}{\sqrt{p(1-p)}}
\end{equation}
and the local interaction function
\begin{equation}
 \eta_k(x_1,\ldots,x_k)
 =\E_z\left[\prod_{i=1}^k\kappa(\ip{z}{x_i})\right],
 \label{eq:eta}
\end{equation}
where $z$ is independent and uniform on $\mathbb S^{d-1}$.
For any square-integrable function $D$ of latent
points, define the second moment of its posterior mean by
\begin{equation}
 \Post(D)
 :=\E\left[\left(\E[D\mid G_n]\right)^2\right].
 \label{eq:post}
\end{equation}

\paragraph{Declaration of AI use.}
The initial proof was discovered by GPT-5.6 under the authors' guidance and was substantially simplified and reorganized by the authors with AI assistance.
A detailed account is
provided in the \hyperlink{sec:AI_statement}{AI methodologies statement} at the end of the paper.
The authors verified the correctness and originality of all content, including references.

\section{Proof of Theorem~\ref{thm:main-detection}} \label{sec:pf-main-thm}

The proof proceeds in five main steps, which we outline below.
\begin{enumerate}
 \item Following~\cite{brennan2020phase,liu2022testing,du2026resolution}, we reveal the graph one
 vertex at a time. The KL chain rule and the data processing inequality reduce the problem to analyzing
 the truncated $\chi^2$-divergence $\Delta$ for the neighborhood of the $n$th vertex conditional on $G_{n-1}$: 
 \begin{equation}
 \KL(\Pbb_{\mathrm{RGG}}\Vert\Pbb_{\mathrm{ER}})
  \le (n-1) \Delta + o(1),
 \label{eq:outline-KL}
 \end{equation}
where the $o(1)$ error is incurred by the truncation to a typical-degree event. 

 \item 
We compute the Fourier expansion of the conditional likelihood ratio for the neighborhood of the $n$th vertex, where the coefficients are posterior means of the local interaction functions
 $\eta_k$ defined in \eqref{eq:eta}, leading to
 \begin{align}
  \Delta\le 2\sum_{k=2}^{K-1}\binom{n-1}{k}\Post(\eta_k)+o(1/n), \label{eq:Delta_expansion}
 \end{align} 
for sufficiently large fixed
 $K=K(\delta)$, 
 where $\Post(\eta_k)$ is defined in \eqref{eq:post}.
 The degree truncation is essential for showing that the high
 Fourier frequencies $k\ge K$ contribute only $o(1/n)$. 

 \item 
For each $k<K$, a spherical-harmonic expansion approximates
 $\eta_k$ in $L_2$ by finitely many spherical harmonic polynomials of total degree
 at most $2M$. We choose a fixed
 $M=M(\delta)$ sufficiently large so that the total contribution of these approximation errors to $\Delta$ is $o(1/n)$.
This is possible because a polynomial of total degree $2r$
 appearing in this expansion has unconditional second moment at most
 $O(d^{-r})$. 
Therefore, $\Delta$ can be bounded in terms of $\Post(D)$ for a fixed finite
 collection of bounded-degree spherical harmonic polynomials $D$ of variables $x_1,\ldots,x_k$. 
 

 \item The cavity argument in Sections~\ref{sec:cavity-recurrence}
 and~\ref{sec:closing-recursion} bounds these posterior moments in terms of
 $\Delta$ itself. For a given $D$, select one of its variables,
 say $x_u$. We first omit the finitely many edges joining $u$ to the
 other vertices in $D$, at a controlled cost. Conditioning on all
 edges not incident to $u$, the Bayes rule expresses the resulting posterior
 mean as a ratio whose denominator is determined by the conditional law of
 the remaining edges incident to $u$. The same quantity $\Delta$ can be used to control the
 contribution due to the denominator being too small. 
When the denominator is bounded from below, 
by averaging over $x_u$ and generalizing the Fourier-harmonic expansion in the previous steps, we relate
 $\Post(D)$ to the posterior moments of other spherical harmonic polynomials. 
This yields a recurrence for the posterior moments of bounded-degree spherical harmonic polynomials.

 \item Substituting the posterior bounds obtained from the cavity recurrence
 into~\prettyref{eq:Delta_expansion} yields the self-bounding inequality
 \[
  \Delta\le o(1/n)+o(1/(np))\Delta,
 \]
resulting in $\Delta=o(1/n)$ since $np\ge n^\delta$.
The KL reduction \prettyref{eq:outline-KL} then gives
$\KL(\Pbb_{\mathrm{RGG}}\Vert\Pbb_{\mathrm{ER}})=o(1)$, completing the
 proof of Theorem~\ref{thm:main-detection}.
\end{enumerate}

\subsection{KL reduction}
\label{sec:KL-reduction}
Throughout this section we denote $n'=n-1$. For any realization of $x_1,\dots,x_{n'}\in \mathbb{S}^{d-1}$ and the corresponding geometric graph $G_{n'}$, we view $x_n\equiv z$ as an independent vertex. The new row of the adjacency matrix is given by 
\begin{equation}
\bY=(y_i)_{i=1}^{n'}
 =\bigl(\one\{\ip{z}{x_i}\ge\tau\}\bigr)_{i=1}^{n'}\,.
\label{eq:bY}
\end{equation}
We will compare $\mathcal L(\bY\mid G_{n'})$, the conditional law of $\bY$ given $G_{n'}$, with the reference measure
$Q=\operatorname{Bern}(p)^{\otimes n'}$. Denote
\[
 \mathcal K
 =\E_{G_{n'}}\left[
  \KL\bigl(\cL(\bY\mid G_{n'})\Vert Q\bigr)
 \right]\,. 
\]
Let $\Lambda$ denote the density of $\cL(\bY\mid G_{n'})$ relative to
$Q$.\footnote{Here, $\Lambda$ implicitly depends on $G_{n'}$, though we suppress this from the notation.} For a 
 suitably large constant $B_*$, define 
\[
 r^+= \lceil 2np+100\log n \rceil + B_*,\qquad
 T^+
 =\left\{y\in\{0,1\}^{n'}:
   \sum_{i=1}^{n'} y_i\le r^+\right\},
\]
and a truncated $\chi^2$-divergence:
\begin{equation}
 \Delta
 =\E_{G_{n'}}\bigg[\int_{T^+}(\Lambda-1)^2\,dQ\bigg]\,.
 \label{eq:Delta-m}
\end{equation}

\begin{theorem}
\label{thm:KL-reduction}
We have
\begin{equation}
 \KL(\Pbb_{\mathrm{RGG}}\Vert\Pbb_{\mathrm{ER}})
 \le n'\mathcal K\le n'\Delta+o(1)\,.
 \label{eq:KL-reduction}
\end{equation}
Consequently, in order to prove Theorem~\ref{thm:main-detection}, it suffices to prove $\Delta=o(1/n)$.

\end{theorem}

Results similar to Theorem~\ref{thm:KL-reduction} have already been established in~\cite[Claim 8.2 and Lemma 8.3]{liu2022testing}. We include a short proof below for completeness. We first record an elementary inequality, bounding KL divergence by a truncated $\chi^2$-divergence:

\begin{lemma}
\label{l:truncated.chi.square}
Let $P,Q$ be probability measures on a finite set $\Omega$, with
$P\ll Q$, and write $\Lambda=dP/dQ$ and
$q_*=\min_{\omega\in\Omega}Q(\omega)$. Then, for every $T\subseteq\Omega$, we have
\[
\KL(P\Vert Q)
\le
\int_T(\Lambda-1)^2\,dQ
+
P(T^c)\log\frac1{q_*}
+
Q(T^c).
\]
\end{lemma}

\begin{proof}
Let $h(x)=x\log x-x+1$. Since $\int(\Lambda-1)\,dQ=0$, we have
$\KL(P\Vert Q)=\int h(\Lambda)\,dQ$. On $T$, the inequality
$\log x\le x-1$ gives $h(\Lambda)\le(\Lambda-1)^2$. Moreover,
$\Lambda\le q_*^{-1}$, so on $T^c$ we have
$h(\Lambda)\le \Lambda\log(q_*^{-1})+1$. Integrating over $T$ and $T^c$
gives the result.
\end{proof}

\begin{proof}[\hypertarget{proof:t.KL-reduction}{Proof of Theorem~\ref{thm:KL-reduction}}] 
We prove the first inequality by KL chain rule and monotonicity. For $1\le m\le n'$, recall that $G_m$ is the
graph induced by the first $m$ latent points, and $z\in \mathbb{S}^{d-1}$ is independent of $G_m$. 
Denote by $A_{ij}$ the indicator of the edge between $i$ and $j$.
For \ER~graph, $A_{ij}$'s are iid as $\Bern(p)$.

Revealing one vertex at a time, the KL divergence between RGG and  \ER~satisfies
\begin{align*}
\KL(\Pbb_{\mathrm{RGG}}\Vert\Pbb_{\mathrm{ER}}) 
\stepa{=} & 
\sum_{m=1}^{n'}
\Expect_{G_m}\KL(\Law(A_{m+1,1:m}\mid G_m) \| \Bern(p)^{\otimes m})    \\
\stepb{=} & 
\sum_{m=1}^{n'}
\Expect_{G_m}
\KL(\Law(A_{n,1:m}\mid G_m) \| \Bern(p)^{\otimes m})    \\
\stepc{\leq} & 
\sum_{m=1}^{n'}
\Expect_{G_{m}}
\KL(\Law(A_{n,1:n'}\mid G_m) \| \Bern(p)^{\otimes n'})    \\
\stepd{\leq} & 
\sum_{m=1}^{n'}
\Expect_{G_{n'}}
\KL(\Law(A_{n,1:n'}\mid G_{n'}) \| \Bern(p)^{\otimes n'})   = n' \calK,
\end{align*}
where 
(a) follows from the chain rule of KL divergence;
(b) is by symmetry;
(c) is by data processing inequality;
(d) is because further conditioning increases the KL divergence (i.e., the  convexity).


Apply Lemma~\ref{l:truncated.chi.square} conditionally on $G_{n'}$ with
$P=\cL(\bY\mid G_{n'})$, $Q=\operatorname{Bern}(p)^{\otimes n'}$, and
$T=T^+$. For all sufficiently large $n$, $p\le 1/2$, so
$q_*=p^{n'}$. Averaging over $G_{n'}$, and using the fact that the
unconditional law of $\bY$ is exactly $Q$, gives
\[
\mathcal K
\le
\Delta+
\bigl(n'\log(1/p)+1\bigr)Q((T^+)^c)\,.
\]
If $X\sim\operatorname{Bin}(n',p)$, then
$Q((T^+)^c)\le\mathbb{P}(X>2np+100\log n+B_*)$. A standard Bernstein
bound gives $Q((T^+)^c)\le n^{-30}$ for all sufficiently large
$n$. Since $p\ge n^{-1+\delta}$ implies
$\log(1/p)=O(\log n)$, it follows that
$
\mathcal K\le \Delta+o(1/n).
$
Combining this with the first inequality yields
\[
\KL(\Pbb_{\mathrm{RGG}}\Vert\Pbb_{\mathrm{ER}})
\le n'\mathcal K
\le n'\Delta+o(1)\,,
\] which concludes the proof. 
\end{proof}


\subsection{Fourier reduction of the row likelihood}
\label{sec:Fourier-reduction}

For $Q=\Bern(p)^{\otimes n'}$, denote the Fourier basis functions~\cite{o2014analysis}
\begin{equation}
    \psi_S(y)
    \equiv \psi_S(y_S) 
    \equiv \prod_{i\in S} \psi(y_i)
    =\prod_{i\in S}\frac{y_i-p}{\sqrt{p(1-p)}}\,,
    \label{eq:fourierbasis}
\end{equation}
indexed by subsets $S\subseteq[n']$.
Denote the likelihood ratio of $\Law(\bY\mid x_1,\ldots,x_{n'},z)$ with respect to 
the reference measure $Q$ by
$\Lambda_{x_{1:n'},z}$.
Since $\bY$ is a deterministic function \prettyref{eq:bY} of $x_{1:n'}$ and $z$, recalling the definition \prettyref{e:def.kappa} of $\kappa$, it is straightforward to calculate the Fourier coefficients of $\Lambda_{x_{1:n'},z}$ in the $\psi_S$ basis, resulting in
	\[
	\Lambda_{x_{1:n'},z}(y)
	= \sum_{S\subseteq[n']}
	\bigg(\prod_{i\in S} \kappa(\ip{z}{x_i})\bigg) \psi_S(y)\,.
	\]
Averaging over $z$, we obtain the expansion of the likelihood ratio of $\Law(\bY\mid x_1,\ldots,x_{n'})$ over $Q$ as
	\[
	\Lambda_{x_{1:n'}}(y)
= \sum_{S\subseteq[n']}
\underbrace{\E_z\bigg[
		\prod_{i\in S} \kappa(\ip{z}{x_i})\bigg]}_{\equiv \eta_S(x_S)} \psi_S(y)
	\]
where $\eta_S$ is the same as the function $\eta_{|S|}$ of \eqref{eq:eta} evaluated at the points
indexed by $S$.
Further averaging over $x_1,\ldots,x_{n'}$ conditional on $G_{n'}$, we obtain the Fourier expansion
of the conditional row likelihood ratio $\Lambda$:
\begin{equation}
 \Lambda(y)-1
 =\sum_{\varnothing\ne S\subseteq[n']}
   \E[\eta_S\mid G_{n'}]\psi_S(y)\,.
 \label{eq:row-Fourier}
\end{equation}
Using this Fourier expansion, the following result bounds the truncated $\chi^2$-divergence $\Delta$ appearing in Theorem~\ref{thm:KL-reduction}. It should be compared with \cite[Lemma 5]{mao2026random}, the key difference being that Theorem~\ref{thm:Fourier-reduction} gives control over the high-frequency Fourier contribution. We present its proof below, assuming some results which are deferred to Section~\ref{sec:high-level-fourier}:

\begin{theorem}
\label{thm:Fourier-reduction}
For any  fixed integer $K>2/\delta$,
the quantity $\Delta$ defined by
\eqref{eq:Delta-m} satisfies the bound
\begin{equation}
 \Delta
 \le 2 \sum_{k=2}^{K-1}\binom {n'}k\Post(\eta_k)+o(1/n)\,,
 \label{eq:Fourier-reduction}
\end{equation}
with $\Post$ defined by \eqref{eq:post}.\footnote{The quantity $\Post(\eta_k)$ appearing in \eqref{eq:Fourier-reduction} is the same as the quantity $g(k)$ appearing in \cite[Lemma 5]{mao2026random} and \cite[Equation~(2)]{du2026resolution}.}
\end{theorem}

\begin{proof}
We bound the truncated $\chi^2$-divergence via the Fourier expansion of $\Lambda-1$, split into frequencies below and above a fixed cutoff $K$. The low-frequency part is controlled by enlarging the integral from $T^+$ to the full cube and applying Parseval's identity.
For the high-frequency part, it is crucial to retain the truncation to $T^+$, and
the resulting contribution is $o(1/n)$ for a sufficiently large fixed $K=K(\delta)$: this is the content of Propositions~\ref{p:high-Fourier-tail} and \ref{p:bounding-B} in Section~\ref{sec:high-level-fourier}.

More precisely, using the exact Fourier expansion \eqref{eq:row-Fourier} and the fact that $ \eta_{\{i\}}\equiv0$ for any $i\in[n']$, we can decompose $\Lambda-1=U_K+V_K$ where $U_K$ and $V_K$ capture respectively the low-frequency and high-frequency Fourier contributions:
\begin{align*}
U_K &:=\sum_{2\le |S|<K}\E[\eta_S\mid G_{n'}]\psi_S\,,\\
V_K &:=\sum_{|S|\ge K}\E[\eta_S\mid G_{n'}]\psi_S\,.
\end{align*}
Then, again recalling the definition \eqref{eq:Delta-m} of $\Delta$, we can bound
\[
\Delta
\le
2\E_{G_{n'}}\bigg[\int_{T^+}
    (U_K)^2\,dQ\bigg]
+
2\E_{G_{n'}}\bigg[\int_{T^+}
    (V_K)^2\,dQ\bigg]\,.
\]
For the first term, dropping the $T^+$ truncation and using Parseval's identity gives
\[
\E_{G_{n'}}\bigg[\int_{T^+}
    (U_K)^2\,dQ\bigg]
\le
\E_{G_{n'}}\bigg[\int (U_K)^2\,dQ\bigg]
=
\sum_{2\le |S|<K}
\E_{G_{n'}}\bigg[
    \E(\eta_S\mid G_{n'})^2\bigg]
\le \sum_{k=2}^{K-1}
    \binom {n'}k\Post(\eta_k)\,,
\]
where the last inequality follows by exchangeability and the conditional Jensen's inequality. (Note that this argument is essentially the same as that in \cite[Lemma~5]{mao2026random}.) 
For the high-frequency term, apply Proposition~\ref{p:high-Fourier-tail} in Section~\ref{sec:high-level-fourier} with $m=n'$, $h=0$, $D\equiv 1$ and $\mathcal F=\sigma(G_{n'})$. Combining with Proposition~\ref{p:bounding-B} in Section~\ref{sec:high-level-fourier}, and choosing $K>2/\delta$, we obtain
\[
\E_{G_{n'}}\bigg[\int_{T^+}(V_K)^2\,dQ\bigg]=o(1/n)\,.
\]
Combining the preceding estimates gives the claim.
\end{proof}

\subsection{Harmonic expansion of the local interactions}
\label{sec:harmonic-expansion}

For $\ell\ge0$, let $\calH_{\ell,d}$ be the space of degree-$\ell$
spherical harmonics on $\mathbb S^{d-1}$ and let
$N_{\ell,d}=\dim(\calH_{\ell,d})$.  Fix an orthonormal basis
$\{Y_{\ell,d,j}:j\in[N_{\ell,d}]\}$ of each $\calH_{\ell,d}$, and write
\[
 Y_{\ell,d}(x)
 =\bigl(Y_{\ell,d,j}(x):j\in[N_{\ell,d}]\bigr).
\]
For the remainder of this subsection, suppress the ambient-dimension
subscript and write
\[
 \calH_\ell:=\calH_{\ell,d},\qquad
 N_\ell:=N_{\ell,d},\qquad
 Y_{\ell,j}:=Y_{\ell,d,j},\qquad
 Y_\ell:=Y_{\ell,d}.
\]
Let $V\subset[n]$ and let
$\bsell=(\ell_v:v\in V)$, with $\ell_v\ge1$. Consider the tensor-product space of spherical harmonic polynomials:
\begin{equation}
\calH_{\bsell}:=
 \bigotimes_{v\in V}\calH_{\ell_v}
 :=\operatorname{span}\left\{
   \prod_{v\in V}f_v(x_v):
   f_v\in\calH_{\ell_v}\text{ for every }v\in V
 \right\}.
\label{eq:tensor-harmonics}
\end{equation}
(In the notation $\calH_{\bsell}$ we suppress the set $V$ of variables as it is encoded in the domain of $\bsell: V\to\naturals$.)
Since the functions
\[
 \prod_{v\in V}Y_{\ell_v,j_v}(x_v),
 \qquad
 \mathbf j=(j_v:v\in V)
 \in\prod_{v\in V}[N_{\ell_v,d}],
\]
form an orthonormal basis of
$\calH_{\bsell}$, every
$D(x_V)$ in this tensor-product space has the unique representation
\[ 
 D(x_V)
 =\sum_{\mathbf j}
 c_{\mathbf j}\prod_{v\in V}Y_{\ell_v,j_v}(x_v).
\]
Equivalently, $D$ can be written in the tensor inner-product form
\begin{equation}
 \label{eq:pure-block}
 D(x_V)=\left\langle \mathsf C_D,
      \bigotimes_{v\in V}Y_{\ell_v}(x_v)\right\rangle,
 \qquad \ell_v\ge1,
\end{equation}
where $\mathsf C_D=(c_{\mathbf j})_{\mathbf j}$ is the coefficient tensor of $D$.

For a nonzero $D\in \calH_{\bsell}$, we call $\bsell$ its \emph{multidegree} and
$\|\bsell\|_1$ its \emph{total degree}.\footnote{
This is well-defined because for fixed $V$ and distinct $\bsell\neq \bsell'$, 
we have $\calH_{\bsell} \perp \calH_{\bsell'}$ and $\calH_{\bsell} \cap \calH_{\bsell'} =\{0\}$, which follows from the orthogonality of $\calH_\ell$ for different $\ell$. Thus, every nonzero $D\in \cup \calH_\ell$ is contained in a unique $\calH_{\bsell}$.}
Note that $\calH_{\bsell}$'s are orthogonal for different multidegree $\bsell$.

For $k \ge 2$ and $q \ge 1$, define the subset of indices
\begin{equation}\label{e:def.calA}
\calA_{k,q}
:=
\left\{
\mathbf a\in\naturals^k:
\|\mathbf a\|_1=2q 
\right\}.
\end{equation}
The following result quantifies the orthogonal projection of $\eta_k$ 
as a function of $x_1,\ldots,x_k$ onto the direct sum of $\calH_{\bfa}$ for total degree $\|\bfa\|_1$ up to $2M$.

\begin{theorem}
\label{thm:harmonic-expansion}
Fix integers $K\ge3$ and $M\ge1$. For every $2\le k<K$, the quantity
$\eta_k$ defined by \eqref{eq:eta} has the orthogonal expansion
\begin{equation}
 \eta_k
 =\sum_{{ q}=\lceil k/2\rceil}^{M}
   \sum_{\bfa\in\calA_{k,q}}
   \gamma_{\bfa}D_{\bfa}
   +\cR_{k,M}
 \label{eq:harmonic-expansion}
\end{equation}
with $D_{\bfa}\in\calH_{\bfa}$.
There are constants $C_0,C_{K,M}>0$, such that
if $d\ge3$ and $0<p\le e^{-1}$, then for each multidegree
$\bfa$ with $\|\bfa\|_1=2q \le 2M$, we have
\begin{equation}
 |\gamma_{\bfa}|^2\le C_{K,M}p^kL^{2  q},\qquad
 \E[(D_{\bfa})^2]\le 
 \frac{C_{K,M}}{d^{ q}}\,.
 \label{eq:harmonic-component-scales}
\end{equation}
The only contribution to \eqref{eq:harmonic-expansion} from
  ${ q}=1$ is
 \begin{equation}
  k=2,\qquad \bfa=(1,1),\qquad
  D_{(1,1)}(x_1,x_2)=\ip{x_1}{x_2}\,.
  \label{eq:critical-component}
 \end{equation}
If, in addition, $d\ge C_0L^2$, the harmonic remainders
$\mathcal R_{k,M}$ satisfy
\begin{equation}
 \E[(\mathcal R_{k,M})^2]
 \le C_{K,M}\frac{p^kL^{2M+2}}{d^{M+1}},
 \label{eq:initial-harmonic-tail}
\end{equation}
where the expectation is over $x_1,\ldots,x_k$ sampled independently and uniformly from $\mathbb{S}^{d-1}$.
\end{theorem}

The \hyperlink{proof:thm.harmonic-expansion}{proof of~\prettyref{thm:harmonic-expansion}} is deferred to~\prettyref{sec:high-harmonic-remainder}.

\subsection{Posterior bounds for spherical harmonics}
\label{sec:posterior-components}

Recall the definition of $L$ from \eqref{e:def.L.eps}. For $t\ge1$, define
\begin{equation}
 \nu_1=1,\qquad \nu_t=2t\quad(t\ge2),\qquad
 g_t=\frac{pL^{\nu_t}}{d^t}.
 \label{eq:gt}
\end{equation}
Thus $g_1=pL/d$ is the sharper target scale for the unique degree-two spherical harmonic, whereas $g_t=pL^{2t}/d^t$ for every $t\ge2$. The following is the principal technical result:

\begin{theorem}
\label{thm:posterior-components}
For all sufficiently large constants $K$ and $M$ (depending on $\delta$), 
there is a constant $C=C(\delta,K,M)$ such that for
every $D_{\bfa}$ in~\eqref{eq:harmonic-expansion} of total degree $\|\bfa\|_1=2q$, we have
\begin{equation}
 \Post(D_{\bfa})\le Cg_{ q}
 \bigg(1+ \frac{\Delta}{p}\bigg)
 \label{eq:component-posterior}
\end{equation}
for $\Delta$ as defined by \eqref{eq:Delta-m}, and $g_{ q}$ as defined by \eqref{eq:gt}. 
\end{theorem}

The \hyperlink{proof:t.posterior-components}{proof of Theorem~\ref{thm:posterior-components}} is given in  Section~\ref{sec:closing-recursion}.

\subsection{Proof of main result}
\label{sec:complete-main-proof}

\begin{proof}[\hypertarget{proof:main}{Proof of Theorem~\ref{thm:main-detection}}]
Recall from Theorem~\ref{thm:KL-reduction} that  it suffices to prove $\Delta=o(1/n)$, where $\Delta$ has the expansion \eqref{eq:Fourier-reduction} from Theorem~\ref{thm:Fourier-reduction}. We will bound the terms in this expansion using Theorems~\ref{thm:harmonic-expansion}~and~\ref{thm:posterior-components}. 
Define the normalized maximum over $D_{\bfa}$ in~\eqref{eq:harmonic-expansion}  by
\[
 Z_{\mathrm{init}}
 :=
  \max_{\substack{2\le k<K,\ 1\le q\le M\\
                  \bfa\in\calA_{k,q}}}
  \frac{\Post(D_{\bfa})}{g_q}\,,
\]
where we recall that $\calA_{k,q}$ is defined by \eqref{e:def.calA}. Theorem~\ref{thm:posterior-components} implies
$Z_{\mathrm{init}}\le C(1+\Delta/p)$.  Note that the total number of multidegrees $\bfa$ contributing to the above maximum is bounded by a constant depending only on $K,M$. By Theorem~\ref{thm:harmonic-expansion}, the Cauchy--Schwarz inequality,
and the conditional Jensen's inequality, we have
\[
 \Post(\eta_k)
 \le C\sum_{q,\bfa} (\gamma_{\bfa})^2
          \Post(D_{\bfa})
      +C\E[(\cR_{k,M})^2]\,.
\]
For $q=1$, it follows from \eqref{eq:critical-component} of Theorem~\ref{thm:harmonic-expansion} that we must have $k=2$ and $\bfa=(1,1)$. The corresponding contribution to
\eqref{eq:Fourier-reduction} is bounded by
\[
 C \cdot \binom{n'}{2} 
 \cdot (\gamma_{(1,1)})^2
 g_1 Z_{\mathrm{init}}
\le
 Cn^2 \cdot 
 p^2L^2
 \cdot \frac{pL}{d}Z_{\mathrm{init}}
 =\frac{C\varepsilon_n}{n}Z_{\mathrm{init}} = o(1/n) Z_{\mathrm{init}}
 \,,
\]
having used \eqref{eq:gt} 
together with the first part of \eqref{eq:harmonic-component-scales} from Theorem~\ref{thm:harmonic-expansion}.
For $q\ge2$, the corresponding contribution to
\eqref{eq:Fourier-reduction} is bounded by
\begin{align*}
C \cdot \binom{n'}{k} \cdot (\gamma_{\bfa})^2
g_q Z_{\mathrm{init}}
\le
 Cn^k \cdot p^kL^{2q}
 \cdot \frac{pL^{2q}}{d^q}Z_{\mathrm{init}}
 &=\frac Cn\,
   \frac{(\varepsilon_n)^q L^q}{(np)^{3q-k-1}}
   Z_{\mathrm{init}}=o(1/n) Z_{\mathrm{init}} \,,
\end{align*}
where the last estimate holds
because Theorem~\ref{thm:harmonic-expansion} only includes
$q\ge\lceil k/2\rceil$, and consequently 
$3q-k-1\ge q-1\ge1$.

By assumption~\eqref{eq:regime}, we have $d\gg L^2$. The remainder
bound~\eqref{eq:initial-harmonic-tail}
from Theorem~\ref{thm:harmonic-expansion} implies that
the aggregate contribution of $\E[(\cR_{k,M})^2]$ to
\eqref{eq:Fourier-reduction} is at most
\begin{align*}
  C_{K,M}\sum_{k=2}^{K-1}
 \frac{(np)^kL^{2M+2}}{d^{M+1}} 
 \ll \sum_{k=2}^{K-1}(np)^{-(3M+3-k)}L^{-(M+1)}
 =o(1/n)\,,
\end{align*}
where the first inequality follows from $d\gg(npL)^3$; the last equality holds by
$ np \ge n^\delta$ and choosing fixed $M$ sufficiently large so that
$\delta(3M+3-k)>1$ for every $2\le k<K$.  

Altogether, by substituting the above estimates into the expansion
\eqref{eq:Fourier-reduction} from
Theorem~\ref{thm:Fourier-reduction}, we conclude that there exist deterministic
$q_n,b_n\ge0$, such that
\begin{equation}
 \Delta\le q_nZ_{\mathrm{init}}+b_n,
 \qquad q_n=o(1/n),\qquad b_n=o(1/n).
 \label{eq:initial-affine-row}
\end{equation}
Since $Z_{\mathrm{init}}$ is in turn bounded by $ C(1+\Delta/p)$ by
Theorem~\ref{thm:posterior-components}, we obtain
\[
\Delta\le Cq_n+Cq_n\frac{\Delta}{p}+b_n\,.
\]
Rearranging gives
\[
\Delta\le \frac{Cq_n+b_n}{1-Cq_n/p}=o(1/n)\,,
\]
having used the bounds in \eqref{eq:initial-affine-row} together with the fact that $p\ge n^{-1+\delta}$, which implies $q_n/p=o(1)$. Having proved $\Delta=o(1/n)$, the main conclusion now follows from
Theorem~\ref{thm:KL-reduction}.\end{proof}

\section{Derivation of the cavity recurrence}
\label{sec:cavity-recurrence}

To control the posterior moments of $D_{\mathbf a} \in \calH_{\mathbf a}$ appearing
in Theorem~\ref{thm:harmonic-expansion}, we derive a rigorous version of the cavity recursion from
statistical physics.  The resulting recursion does not preserve multidegree:
starting from $\Post(D_{\mathbf a})$, it introduces posterior moments of spherical harmonics
with other multidegrees.  The finite collection appearing in
Theorem~\ref{thm:harmonic-expansion} is therefore not closed under the
recursion, so here we derive the recurrence for a general multidegree
spherical harmonic. The bounded-degree class needed to close the recurrence will be introduced in the next section.





Fix $V \subset [n]$ and a positive multidegree $\bsell=(\ell_v:v\in V)$ with $\bsell_v>0$ with the total degree even. Pick any nonzero spherical harmonic polynomial $D\equiv D(x_V) \in \calH_{\bsell}$,   as in \eqref{eq:pure-block}.   Let $t(D)=\|\bsell\|_1/2$ denote half of its total degree.  

Define 
\begin{equation} \label{eq:def-nu-D}
\nu(D):=
\begin{cases}
1, & \text{if } D= c \iprod{x_u}{x_v}\text{ for some $c \neq 0$,} \\
2t(D) = \|\bsell\|_1, & \text{otherwise.}
\end{cases}
\end{equation}
(Note in the first case that $2t(D)=2$. This should be compared with the definition of $\nu_t$ in \eqref{eq:gt}.) Recalling the definition of $L$ from 
\eqref{e:def.L.eps}, we also define
\begin{equation}
g(D):=\E(D^2)pL^{\nu(D)}\,.
\label{eq:gD}
\end{equation}
Thus $g(cD)=c^2g(D)$ for every scalar $c$, while $\nu(D)$ is invariant under
nonzero rescaling. As we will see, $g(D)$ is the target scale for $\Post(D)$.

We first provide a generic bound on the truncated second moment of $D$, by using H\"older's inequality to decouple $D$ and the indicator. The higher moments of the polynomial $D$ are controlled in terms of its second moment using hypercontractivity.

\begin{lemma}[truncated second moment of spherical harmonics]
\label{lem:uniform-localization}
There exists a constant $C$ 
depending only on the total degree $2 t(D)$ such that for every event
$E\in\sigma(x_1,\ldots,x_n)$ with $0<\Pbb(E)\le1/2$,
\begin{equation}
 \E(D^2\one_E)
 \le C\E(D^2)\Pbb(E)
    \log^{\nu(D)}\!\bigg(
    \frac e{\Pbb(E)}\bigg)\,,
 \label{eq:state-event-moment}
\end{equation}
where $\Pbb$ is the product spherical prior of the latent points.
\end{lemma}

\begin{proof}
We first claim that,
for every $q\ge2$, 
\begin{equation}
 \|D\|_q\le(q-1)^{t(D)} \|D\|_2\,.
 \label{eq:state-hypercontractivity}
\end{equation}
For each $v\in V$, let $g_v\sim N(0,I_d)$ be an independent standard gaussian vector and write $x_v=g_v/r_v$ with $r_v=\|g_v\|_2.$
The random vectors $x_v$ are independent and uniform on
$\mathbb S^{d-1}$, and $(x_v:v\in V)$ is independent of
$(r_v:v\in V)$.
The homogeneous extension 
\[
\widetilde D((g_v)_{v\in V})
=\prod_{v\in V} r_v^{\ell_v} \cdot D((x_v)_{v\in V})
\]
is a gaussian polynomial of total degree $2 t(D)$: in particular, 
$r_v^{\ell_v}Y_{\ell_v}(x_v)$ is a homogeneous harmonic polynomial
of degree $\ell_v$ in $g_v$. 
The independence between $x_V$ and $r_V$ gives 
$$
 \|\widetilde D\|_q
 =\|D\|_q \prod_{v\in V} \| r_v^{\ell_v}\|_q
$$
Moreover, gaussian
hypercontractivity \cite[Theorem~5.10 and Remark~5.11]{janson1997gaussian} gives
\[
 \|\widetilde D\|_q
 \le(q-1)^{t(D)}\|\widetilde D\|_2.
\]
Combining the last two displayed equations yields
\[
 \|D\|_q 
 \le (q-1)^{t(D)}\|D\|_2
 \prod_{v\in V}
 \frac{\|r_v^{\ell_v}\|_2}{\|r_v^{\ell_v}\|_q}
 \le (q-1)^{t(D)}\|D\|_2,
\]
where the last inequality holds
due to
$\|r_v^{\ell_v}\|_2\le\|r_v^{\ell_v}\|_q$ by the  monotonicity of $L^q$ norms.  This proves \eqref{eq:state-hypercontractivity}.

Suppose first that $D$ is not a scalar multiple of $\iprod{x_u}{x_v}$. 
For $\eta:=\Pbb(E)\le1/2$, H\"older's inequality with
$q=2s=2\log(e/\eta)$ therefore gives
\[
 \E(D^2\one_E)
 \le \|D\|_{2s}^2 \eta^{1-1/s}
 \overset{\eqref{eq:state-hypercontractivity}}{\le} (2s-1)^{2t(D)}\E(D^2)\eta^{1-1/s}\le C \E(D^2) \eta \log^{2t(D)}(e/ \eta)
 \,,\]
where $\eta^{-1/s}\le e$. This gives
\eqref{eq:state-event-moment} in the case $ \nu(D)\ge2$.

For the case that $D$ is a scalar multiple of $\iprod{x_u}{x_v}$, we can assume without loss that $D=\langle x_u,x_v\rangle$. 
By rotational invariance, $\langle x_u,x_v\rangle$ has the same distribution as $\langle x_u,e_1\rangle$, a degree-one spherical harmonic. Therefore, in this case, we have the sharper estimate
	\[\|D\|_q\le { (q-1)^{1/2}}\,\|D\|_2\,,\]
by applying~\prettyref{eq:state-hypercontractivity} in one coordinate. Proceeding similarly as before, again with $q=2s=2\log(e/\eta)$, gives
	\[
	\E(D^2\one_E)
	\le \|D\|_{2s}^2 \eta^{1-1/s}
	\le (q-1)\E(D^2) \eta^{1-1/s}
	\le C \E(D^2) \eta \log(e/\eta)\,,
	\]
again using $\eta^{-1/s}\le e$. This gives \eqref{eq:state-event-moment} in the case $\nu(D)=1$, and concludes the proof.
\end{proof}


\subsection{Step 1: dropping internal edges}

Recall $D\equiv D(x_V) \in \calH_{\bsell}$ as in \eqref{eq:pure-block}. To carry out the cavity calculation, we mark an arbitrary vertex $u \in V$ and aim to average out $x_u$. 
Set
\begin{equation}\label{e:def.I.J}
 I=V^c=[n]\setminus V,\qquad
 J= [n]\setminus\{u\}.
\end{equation}
Let $\bG_{-u}=(A_{ij}:i<j, i\neq u, j \neq u)$ be the subgraph induced by all vertices except $u$. (This is an RGG with $n-1$ vertices.)

Denote by
$\bY_I=(A_{ui}: i\in V^c)$ the edges between $u$ and vertices outside $V$, and
$\bZ:=(A_{uv}: v\in V\setminus \{u\})$ the edges between $u$ and other vertices in $V$.
We refer to $\bZ$ and $\bY_I$ the \textit{internal edges} and \textit{external edges}, respectively.
The full graph $G_n$ consists of $\bG_{-u}$, $\bY_I$ and $\bZ$.
This decomposition is illustrated in \prettyref{fig:cavity-vertex-decomposition}.

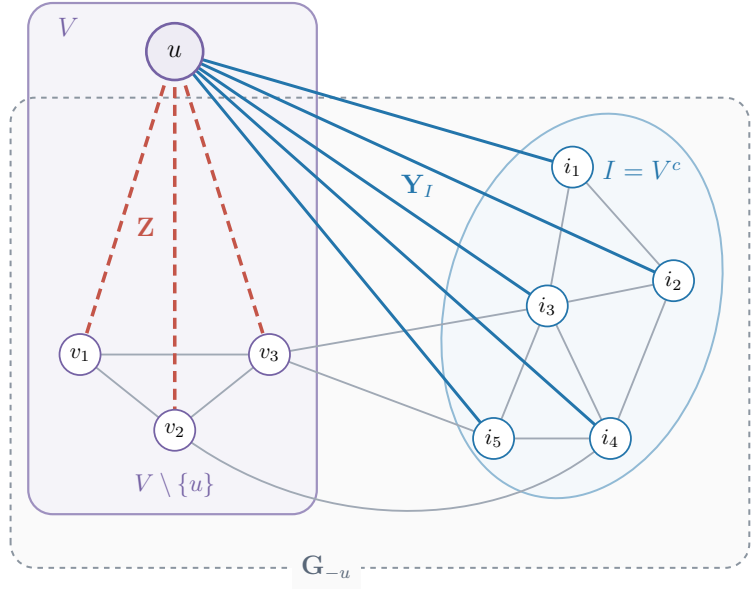
\begin{figure}[!htb]
\centering
\begingroup
\definecolor{ink}{HTML}{243247}
\definecolor{muted}{HTML}{596778}
\definecolor{support}{HTML}{7562A4}
\definecolor{external}{HTML}{2375AB}
\definecolor{dropped}{HTML}{C45549}
\definecolor{backgroundedge}{HTML}{A0A9B5}
\resizebox{0.6\linewidth}{!}{%
\begin{tikzpicture}[
  x=1cm,y=1cm,
  font=\sffamily\color{ink},
  vertex/.style={circle,draw=ink,fill=white,line width=.8pt,
    minimum size=6.5mm,inner sep=0pt,font=\small},
  internal/.style={draw=dropped,line width=1.6pt,dash pattern=on 5pt off 3pt},
  externaledge/.style={draw=external,line width=1.35pt},
  backgroundedge/.style={draw=backgroundedge,line width=.85pt},
  every node/.style={outer sep=0pt}
]

\fill[black!2,rounded corners=11pt] (.20,-1.75) rectangle (11.90,5.70);
\fill[support!7,rounded corners=11pt] (.48,-.90) rectangle (5.05,7.20);
\draw[support!70,line width=1pt,rounded corners=11pt]
  (.48,-.90) rectangle (5.05,7.20);
\draw[external!50,fill=external!5,line width=.9pt,rotate around={-15:(9.25,2.40)}]
  (9.25,2.40) ellipse[x radius=2.15,y radius=3.10];
\draw[muted!70,line width=.9pt,dash pattern=on 2.5pt off 2.5pt,rounded corners=11pt]
  (.20,-1.75) rectangle (11.90,5.70);

\node[anchor=west,text=support,font=\large] at (.82,6.82) {\(V\)};
\node[font=\normalsize,text=support] at (2.80,-.43) {\(V\setminus\{u\}\)};
\node[font=\large,text=external] at (10.20,4.55) {\(I=V^c\)};

\coordinate (u) at (2.80,6.43);
\coordinate (v1) at (1.30,1.63);
\coordinate (v2) at (2.80,.43);
\coordinate (v3) at (4.30,1.63);
\coordinate (i1) at (9.10,4.60);
\coordinate (i2) at (10.70,2.80);
\coordinate (i3) at (8.70,2.40);
\coordinate (i4) at (9.70,.30);
\coordinate (i5) at (7.85,.30);

\draw[backgroundedge] (v1)--(v2)--(v3)--(v1);
\draw[backgroundedge] (v3)--(i5);
\draw[backgroundedge] (v3)--(i3);
\draw[backgroundedge] (v2) .. controls (4.90,-1.20) and (7.90,-1.30) .. (i4);
\draw[backgroundedge] (i1)--(i2)--(i3)--(i1);
\draw[backgroundedge] (i2)--(i4)--(i5)--(i3)--(i4);

\foreach \v in {v1,v2,v3} \draw[internal] (u)--(\v);
\foreach \i in {i1,i2,i3,i4,i5} \draw[externaledge] (u)--(\i);

\node[vertex,fill=support!12,draw=support,line width=1.2pt,
  minimum size=9mm,font=\large] at (u) {\(u\)};
\foreach \v/\label in {v1/v_1,v2/v_2,v3/v_3}
  \node[vertex,draw=support] at (\v) {\(\label\)};
\foreach \i/\label in {i1/i_1,i2/i_2,i3/i_3,i4/i_4,i5/i_5}
  \node[vertex,draw=external] at (\i) {\(\label\)};

\node[text=dropped,font=\large] at (2.34,3.65) {\(\mathbf Z\)};
\node[text=external,font=\large] at (6.65,4.30) {\(\mathbf Y_I\)};
\node[text=muted,font=\large,fill=black!2,inner sep=4pt]
  at (5.20,-1.75) {\(\mathbf G_{-u}\)};
\end{tikzpicture}
}
\endgroup
\caption{Decomposition of $G_n$ into 
$\bZ$ (internal edges, dashed red), 
$\bY_I$ (external edges, solid blue), and $\bG_{-u}$.}
\label{fig:cavity-vertex-decomposition}
\end{figure}

The following lemma shows that the posterior moment of $D$ is nearly unchanged if we drop the conditioning on the internal edges $\bZ$:
\begin{lemma}[dropping internal edges]
\label{lem:restoration}
Assume $|V|=O(1)$. Then there exists a constant $C$, depending only on $|V|$ and $t(D)$, such that
\begin{equation}
 \Post(D)
 \le Cg(D)+4\E\bigg[\Big(\E[D\mid\bG_{-u},\bY_I]\Big)^2\bigg]\,,
 \label{eq:restoration-roadmap}
\end{equation}
with $\Post$ as defined by \eqref{eq:post},
and $g(D)$ the target scale as defined by \eqref{eq:gD}.
\end{lemma}

\begin{proof} 
Let $E_{\bZ}=\{\bZ\ne0\}$. It follows by a union bound that  
	\[\Pbb(E_{\bZ})\le (|V|-1)p
	\le O(p)
	\,.\]   Let $\mathcal{F}^-=\sigma(\bG_{-u},\bY_I)$, and abbreviate $P_z := \Pbb(\bZ=z\,|\,\mathcal{F}^-)$, 
	$A_z := \E(D\,|\,\mathcal{F}^-,\bZ=z)$. Then
	\begin{equation}\label{e:D.condmean.Fminus}
	\E(D\,|\,\mathcal{F}^-)
	= \sum_z P_z A_z
	\equiv P_0 A_0 + 
	\E(D \one_{E_\bZ} \,|\,\mathcal{F}^-)\,.
	\end{equation}
Define the event
	\[E_{\rm int}
	:= \bigg\{
	1-P_0 =
	\Pbb(E_{\bZ} \,|\, \mathcal{F}^-)
	> \frac12
	\bigg\}\,;\]
it follows by Markov's inequality that   $\Pbb(E_{\rm int}) \le 2\Pbb(E_{\bZ}) \le O(p)$.   Recalling \eqref{eq:post}, we can bound the contribution to $\Post(D)$ from outside $E_{\rm int}$ by
	\begin{align*}
	&\E\Big[
	\one_{(E_{\rm int})^c}
	\E(D\,|\,\mathcal{F}^-, \bZ)^2
	\Big]
	= \E\bigg[
	\one_{(E_{\rm int})^c}
	\sum_z P_z (A_z)^2
	\bigg] \\
	&\qquad\le 2
	\E\Big[
	\one_{(E_{\rm int})^c}
	( P_0 A_0)^2
	\Big]
	+ \E\Big[
	\one_{E_\bZ}
	\E(D\,|\, \mathcal{F}^-, \bZ)^2
	\Big]\\
	&\qquad
	\stackrel{\eqref{e:D.condmean.Fminus}}{\le}
	4\E\bigg[
	\E(D\,|\,\mathcal{F}^-)^2
	+  \E(D\one_{E_\bZ} \,|\,\mathcal{F}^-)^2
	\bigg]
	+ \E\bigg[
	\one_{E_\bZ}
	\E(D\,|\, \mathcal{F}^-, \bZ)^2
	\bigg]\\
	&\qquad\le 
	4 \E[\E(D\,|\,\mathcal{F}^-)^2]
	+ 5 \E(D^2 \one_{E_{\bZ}})\,.
	\end{align*}
In the above, the first bound uses that $P_0 \ge 1/2$ outside $E_{\rm int}$, and the last bound uses the conditional Jensen's inequality. Combining with the contribution to $\Post(D)$ from $E_{\rm int}$ gives
	\[
	\Post(D)
	= \E\Big[
	\E(D\,|\,\mathcal{F}^-, \bZ)^2
	\Big]
	\le 4 \E\Big[ 
	\E(D\,|\,\mathcal{F}^-)^2
	\Big]
	+ 5\E(D^2 \one_{E_\bZ} )
	+ \E(D^2 \one_{E_{\rm int}} )\,,
	\]
again using the conditional Jensen's inequality. Applying \eqref{eq:state-event-moment} to the last two terms gives
$Cg(D)$, and the claim follows.
\end{proof}

That is to say, after incurring an error on the order of the target scale  $g(D)$, the internal edges $\bZ$ no longer appear in the conditioning.

\subsection{Step 2: Bayes rule with respect to external edges}

  Recall $D\equiv D(x_V) \in \calH_{\bsell}$ as in \eqref{eq:pure-block}.
Recall also from \eqref{e:def.I.J} that we mark a vertex $u\in V$, let $\bG_{-u}$ be the subgraph induced by all vertices other than $u$, and let $\bY_I$ be the external edges between $u$ and $I=[n]\setminus V$.   \prettyref{lem:restoration} reduces the
main problem of controlling $\Post(D)$ to the problem of controlling the second moment of the posterior mean $\E(D\mid\bG_{-u},\bY_I)$.
The plan next is to
(a) use the Bayes rule to express this posterior mean as a ratio (see \eqref{eq:Bayes_ratio} below), 
(b) lower bound the denominator, then 
(c) bound the numerator by Fourier expansion. Similarly as before, we work with the reference product Bernoulli measure $Q_I=\operatorname{Bern}(p)^{\otimes I}$.  Conditional on $\bG_{-u}$, let
$\Lambda_{u,I}$ be the likelihood ratio of $\bY_I$ relative to $Q_I$,
\[
\Lambda_{u,I}(y) := \frac{\Prob(\bY_I = y
    \mid \bG_{-u})}{Q_I(y)}\,,
\]
for $y \in \{0,1\}^I$.
Define similarly the $D$-weighted quantity
\begin{equation}
 r_{u,I,D}(y)
 :=\frac{\Expect(
 D(x_V) \indc{\bY_I = y} \mid \bG_{-u}
 )}{Q_I(y)}\,.
 \label{eq:rI}
\end{equation}
On the event $\{\Lambda_{u,I}>0\}$, it follows by the Bayes rule that
\begin{align}
\Expect(D \mid \bY_I = y, \bG_{-u})
 =\frac{
 r_{u,I,D}(y)}{\Lambda_{u,I}(y)}. \label{eq:Bayes_ratio}
\end{align}

\begin{lemma}[Bayes rule for external edges]
\label{lem:puncturing}
The second term of
\eqref{eq:restoration-roadmap} can be re-expressed as
\begin{equation}
 \E\left[\left(\E[D\mid\bG_{-u},\bY_I]\right)^2\right]
 =\E\bigg[\int\frac{r_{u,I,D}(y)^2}{\Lambda_{u,I}(y)}\,dQ_I(y)\bigg]\,,
 \label{eq:puncturing-roadmap}
\end{equation}
with the   integrand   
interpreted as zero when $\Lambda_{u,I}=0$.
\end{lemma}
\begin{proof}
By~\prettyref{eq:Bayes_ratio}, conditionally on $\mathbf{G}_{-u}$, 
\begin{align*}
 \int\frac{r_{u,I,D}(y)^2}{\Lambda_{u,I}(y)}\,dQ_I(y)
 &=\int \left(\E[D\mid\bG_{-u},\bY_I=y]\right)^2
       \Lambda_{u,I}(y)\,dQ_I(y)\\
 &=\E\left[\left(\E[D\mid\bG_{-u},\bY_I]\right)^2\mid\bG_{-u}\right].
\end{align*}
Averaging over $\bG_{-u}$ proves \eqref{eq:puncturing-roadmap}.
\end{proof}

\subsection{Step 3: degree truncation}

   Lemmas
\ref{lem:restoration} and \ref{lem:puncturing} reduce the problem of controlling $\Post(D)$ to the problem of controlling the integral \eqref{eq:puncturing-roadmap}. To this end, introduce the local truncation 
\begin{equation}
 T_I
 =\bigg\{y_I\in\{0,1\}^I:
   \sum_{i\in I}y_i\le
   \min\Big\{
   |I|,\lceil2|I|p+100\log n\rceil
   \Big\}
   \bigg\}.
 \label{eq:local-row-truncation}
\end{equation}
Note that this truncation involves only the external edges $\bY_I$.

We first control the contribution to \eqref{eq:puncturing-roadmap} from outside the set $T_I$ defined by \eqref{eq:local-row-truncation}:

\begin{lemma}[degree truncation]
\label{lem:discarded-row}
There exists a constant $C$ depending only on the total degree $2 t(D)$ such that 
\begin{equation}
 \E\bigg[\int_{(T_I)^c}
       \frac{r_{u,I,D}^2}{\Lambda_{u,I}}\,dQ_I\bigg]
       \le C n^{-20} \E(D^2)\log^{\nu(D)}(en) =o(g(D))\,,
 \label{eq:discarded-row}
\end{equation}
for $g(D)$ as defined by \eqref{eq:gD}.
\end{lemma}
\begin{proof}
In this proof, abbreviate $E\equiv \{\bY_I\notin T_I\}$. By \eqref{eq:Bayes_ratio} and Jensen's inequality,
\begin{align}
 \E\left[\int_{(T_I)^c}\frac{r_{u,I,D}^2}{\Lambda_{u,I}}\,dQ_I\right]
 =\E\left[\left(\E[D\mid\bG_{-u},\bY_I]\right)^2
             \one_E\right]\le
    \E(D^2\one_E)\,.
    \label{eq:bayes_trunction}
\end{align}
Under the product spherical prior of the latent points, conditional on $x_u$, the vectors
$(x_i)_{i\in I}$ are independent and uniform, so the row indicators
$(A_{ui})_{i\in I}$ are independent $\operatorname{Bern}(p)$ random variables.  Bernstein's
inequality and the definition of $T_I$ therefore give
\[\Pbb(E)=
 \Pbb(\bY_I\notin T_I)\le \frac{1}{n^{20}}\,.
\]
Lemma~\ref{lem:uniform-localization} now yields
\[
 \E(D^2\one_E)
 \le C \E(D^2)
 \frac{\log^{\nu(D)}(en)}{n^{20}}
 =o(g(D))\,,
\]
where the last equality holds as $p\ge n^{-1+\delta}$ and $\nu(D)$ is fixed. 
\end{proof}


\subsection{Step 4: controlling the bad denominator}\label{subsec:bad-denominator}

  In view of \prettyref{lem:discarded-row}, it remains to control the contribution to \eqref{eq:puncturing-roadmap} from inside the set $T_I$ defined by \eqref{eq:local-row-truncation}. 
 
Recall from \eqref{e:def.I.J} that $J=[n]\setminus\{u\}$, and define
\begin{equation}
 T_{+,J}
 =\bigg\{y_J\in\{0,1\}^J:
   \sum_{i\in J}y_i\le
   \min\bigl\{|J|,\lceil2|J|p+100\log n\rceil\bigr\} +B_*\bigg\}\,,
 \label{eq:local-row-truncation_slack}
\end{equation}
where $B_*$ is a slack parameter.
If $B_* \ge |V|-1,$
then
\begin{equation}
 T_I\times\{0,1\}^{V \setminus \{u\}}\subseteq T_{+,J}.
 \label{eq:local-to-row-truncation}
\end{equation}
Marginalizing the likelihood $\Lambda_{u,J}$ over $Q_{V\setminus\{u\}}$ and applying Jensen's inequality gives
\begin{align}
 &\E\left[\int_{T_I}(\Lambda_{u,I}-1)^2\,dQ_I\right]
= \E\left[\int_{T_I}(\E_{Q_{V\setminus\{u\}}}(\Lambda_{u,J})-1)^2\,dQ_I\right] \nonumber \\
 &\qquad \le \E\left[\int_{T_I}(\Lambda_{u,J}-1)^2\,dQ_J\right] 
\le \E\left[\int_{T_{+,J}}(\Lambda_{u,J}-1)^2\,dQ_J\right]
 \le \Delta,
 \label{eq:averaged-punctured-comparison}
\end{align}
where the last inequality follows from exchangeability, which allows us to identify $J$ with $[n-1]$, together with the definition of $\Delta$ from \eqref{eq:Delta-m}.

Heuristically, \eqref{eq:averaged-punctured-comparison} suggests that $\Lambda_{u,I}$ is typically very close to $1$, since $T_I$ is a typical event and we expect $\Delta$ to be of order $o(1/n)$, although this has not been proved yet   (as noted in the statement of \prettyref{thm:KL-reduction}, this would imply the main result \prettyref{thm:main-detection}).   This then suggests that we may drop the denominator in \eqref{eq:puncturing-roadmap} and incur only a constant factor loss. This heuristic will be made rigorous through the self-bounding scheme below. 

On $T_I$, we divide according to whether $\Lambda_{u,I}<1/2$ or
$\Lambda_{u,I}\ge1/2$, and first deal with the case of $\Lambda_{u,I}<1/2$.   (From the heuristic discussion above, we expect this to be a rare event.)   For fixed $\nu$, define
    \begin{equation}\label{e:def.phi.nu}
    \phi_\nu(s)
    := \begin{cases}
    0 & s=0\,,\\
    s\log^\nu\!(e/s)
        & 0<s<1\,,\\
    1 & s\ge1\,.
    \end{cases}
    \end{equation}
The following controls the contribution to \eqref{eq:puncturing-roadmap} from $T_I\cap\{\Lambda_{u,I} <1/2\}$:

\begin{lemma}[bad denominator estimate]
\label{lem:denominator}
Assume $|V| \le B_*$, where the slack $B_*$ is defined in \eqref{eq:local-row-truncation_slack}. For every  nonzero $D \in \calH_{\bsell}$,
\begin{equation}
 \E\left[\int_{T_I\cap\{\Lambda_{u,I}<1/2\}}
       \frac{r_{u,I,D}^2}{\Lambda_{u,I}}\,dQ_I\right]
 \le C \Expect(D^2) \phi_{\nu(D)}(\Delta)\,,
 \label{eq:bad-denominator}
\end{equation}
where $C$ only depends on $t(D)$.
\end{lemma}

\begin{proof}
Conditional on $\bG_{-u}$, denote
\[
 E_{\bG_{-u}}:=T_I\cap\bigg\{
    y\in\{0,1\}^I : \Lambda_{u,I}(y)<\frac12 \bigg\}\,,\]
  and write $E:= \{\bY_I \in E_{\bG_{-u}}\}$.  
Analogously to~\prettyref{eq:bayes_trunction}, we have
\begin{align*}
 \E\left[\int_{E_{\bG_{-u}}}
       \frac{r_{u,I,D}^2}{\Lambda_{u,I}}\,dQ_I\right]
 &=\E\left[\left(\E[D\mid\bG_{-u},\bY_I]\right)^2\one_E\right]
 \le\E[D^2\one_E].
\end{align*}
We next bound $\Prob(E)$, where $\Pbb$ denotes the product spherical prior of $(x_1,\ldots,x_n)$.
For $y\in E_{\bG_{-u}}$, we have $|\Lambda_{u,I}(y)-1|>1/2$, and hence
\[
 Q_I(E_{\bG_{-u}})\le
  4\int_{T_I}(\Lambda_{u,I}-1)^2\,dQ_I
  := 4 \delta_{\bG_{-u}}\,,
\]
from which it follows that
\begin{align*}
 \int_{E_{\bG_{-u}}}\Lambda_{u,I}\,dQ_I
 \le \frac12Q_I(E_{\bG_{-u}})
 \le 2\delta_{\bG_{-u}}. 
\end{align*}
Conditional on $\bG_{-u}$,
the law of $\bY_I$ has density $\Lambda_{u,I}$ with respect to $Q_I$, so
\begin{align}
 \Pbb(E)
 =\E\left[\Pbb(\bY_I\in E_{\bG_{-u}}\mid\bG_{-u})\right]=\E\left[\int_{E_{\bG_{-u}}}\Lambda_{u,I}\,dQ_I\right]
 \le2 \E(\delta_{\bG_{-u}}) \le 2\Delta,  \label{eq:bad-event-probability}
\end{align}
where the last inequality holds by
\eqref{eq:averaged-punctured-comparison}.
If $\Prob(E) \le1/2$, Lemma~\ref{lem:uniform-localization} and
\eqref{eq:bad-event-probability} give
\[
 \E(D^2\one_E)
 \le C_{t(D)}\E(D^2)\phi_{\nu(D)}(\Delta)\,.
\]
Here we used the elementary fact that, for fixed $\nu$ and $c$, whenever
$0<s\le ct$ and $s\le1/2$,
$
 s\log^\nu(e/s)\le C_{\nu,c}\phi_\nu(t).
$
If $\Prob(E)>1/2$, then \eqref{eq:bad-event-probability} implies
$\Delta>1/12$; and the conclusion follows by noting that
$\E(D^2\one_E)\le\E(D^2)$, and $\phi_\nu$ is bounded below by an absolute positive constant on
$[1/12,\infty)$.
\end{proof}

Recall from the statement of \prettyref{thm:KL-reduction} that we expect $\Delta=o(1/n)$. If this were the case, it would imply that the bound in \eqref{eq:bad-denominator} is of order 
\[O\bigg(\Expect(D^2)\cdot \frac{\log(n)^{\nu(D)}}{n}\bigg)=o(g(D))\,,\]
for $g(D)$ as defined by \eqref{eq:gD},
as $p \gg 1/n$. 
This heuristic suggests that the contribution
to \eqref{eq:puncturing-roadmap} from the bad denominator event is negligible. Again, this will be made rigorous through the self-bounding scheme. 


\subsection{Step 5: Fourier expansion and its high-frequency remainder} \label{sec:step-5-fourier-high-frequency}

  Recall again that 
 Lemmas
\ref{lem:restoration} and \ref{lem:puncturing} reduce the problem of controlling $\Post(D)$ to the problem of controlling the integral \eqref{eq:puncturing-roadmap}. \prettyref{lem:discarded-row} controls the contribution to \eqref{eq:puncturing-roadmap} from outside the set $T_I$ defined by \eqref{eq:local-row-truncation}. \prettyref{lem:denominator} controls the contribution from the ``bad denominator'' set $T_I\cap \{\Lambda_{u,I} <1/2\}$. It remains therefore to bound the contribution from
$T_I\cap \{\Lambda_{u,I} \ge 1/2\}$.   On this event, the denominator 
  in \eqref{eq:puncturing-roadmap}  
can be replaced by a constant. The numerator can be expanded as
\begin{equation}
 r_{u,I,D}(y_I)
 =\sum_{S\subseteq I}
 \widehat r_{u,I,D}(S)\psi_S(y_I),
 \quad y_I \in \{0,1\}^I\,,
 \label{eq:rI-fourierexp}
\end{equation}
where $(\psi_S)_S$ is the Fourier basis defined by \prettyref{eq:fourierbasis}.
Using \prettyref{eq:rI}, each Fourier coefficient is itself a posterior mean:
\begin{align}
&\widehat r_{u,I,D}(S)
=\Expect_{Q_I}\Big[r_{u,I,D}(\bY_I) \psi_S(\bY_S)\Big]
=\sum_{y_I\in\{0,1\}^I}
 \Expect\Big[D(x_V) \indc{\bY_I = y_I} \psi_S(y_S) \,\Big|\, \bG_{-u}\Big]\nonumber\\
 &\qquad
= \Expect\Big[
D(x_V) \psi_S(\bY_S)
\,\Big|\,\bG_{-u}\Big]
= \Expect\bigg[
D(x_u, x_{V\setminus u}) \prod_{j\in S} \kappa(\langle x_u,x_j \rangle)
\,\bigg|\, \bG_{-u}\bigg]
\,,\nonumber
\end{align}
where the last equality uses \eqref{e:def.kappa} and \eqref{eq:fourierbasis}.
We then average over the position $x_u$ of the marked vertex $u$, using its independence from the  graph $\bG_{-u}$: this gives
\begin{equation}
\widehat r_{u,I,D}(S)
= \Expect(\eta_{D,S}\mid\bG_{-u})
    \label{eq:rI-fouriercoeff}
\end{equation}
where, with $\Expect_{x_u}$ denoting the average over a \emph{uniformly} random $x_u\in\mathbb{S}^{d-1}$, we define
\begin{equation}
 \eta_{D,S}
\equiv \eta_{D,S}(x_{V\setminus u},x_S)
:=\E_{x_u}\bigg[D(x_u,x_{V\setminus u})\prod_{j\in S}
     \kappa(\ip{x_u}{x_j})\bigg]\,.
 \label{eq:eta-DS}
\end{equation}
  The following controls the contribution to \eqref{eq:puncturing-roadmap}  from $T_I \cap\{\Lambda_{u,I}\ge1/2\}$, assuming some results which are deferred to Section~\ref{sec:high-level-fourier}:

\begin{lemma}[Fourier truncation]
\label{lem:weighted-fourier}
For any fixed integer $K>2/\delta$, we have
\begin{align}
 &\E\left[\int_{T_I\cap\{\Lambda_{u,I}\ge1/2\}}
       \frac{r_{u,I,D}^2}{\Lambda_{u,I}}\,dQ_I\right] \le
 4 \sum_{1\le |S|<K}
   \E\left[\left|\E[\eta_{D,S}\mid\bG_{-u}]\right|^2\right]
 +o(g(D)).
 \label{eq:weighted-fourier}
\end{align}

\end{lemma}

\begin{proof}
We argue similarly as in the proof of Theorem~\ref{thm:Fourier-reduction}.
Using the Fourier expansion \eqref{eq:rI-fourierexp}, we decompose $r_{u,I,D}$ as the sum of two terms,
\begin{align*}
r_{u,I,D}^{(<K)}
    &= \sum_{S\subseteq I,|S|<K}
    \widehat r_{u,I,D}(S)\psi_S(y_I)
    \,,\\
r_{u,I,D}^{(\ge K)}
    &=\sum_{S\subseteq I,|S|\ge K}
    \widehat r_{u,I,D}(S)\psi_S(y_I)\,.
\end{align*}
The first term is the ``low-frequency'' contribution, while the second term is the ``high-frequency'' contribution. Since $\Lambda_{u,I}\ge 1/2$ on the domain of integration, we have 
\begin{align}
\E\bigg[\int_{T_I\cap\{\Lambda_{u,I}\ge 1/2\}}
\frac{r_{u,I,D}^2}{\Lambda_{u,I}}\,dQ_I\bigg]
&\le
4\E\bigg[\int_{T_I}
\bigl(r_{u,I,D}^{(<K)}\bigr)^2\,dQ_I\bigg]
+
4\E\bigg[\int_{T_I}
\bigl(r_{u,I,D}^{(\ge K)}\bigr)^2\,dQ_I\bigg]\,.
\label{eq:weighted-fourier-split}
\end{align}
For the low-frequency contribution to \eqref{eq:weighted-fourier-split}, enlarge the domain of integration from $T_I$ to the full
cube $\{0,1\}^I$ and apply Parseval's identity. Since $D$ has positive spherical-harmonic degree
in $x_u$, its mean in $x_u$ vanishes for every fixed $x_{V\setminus\{u\}}$.
As $x_u$ is independent of $\bG_{-u}$, this gives
\[
\widehat r_{u,I,D}(\emptyset)
=
\E[D\mid\bG_{-u}]
=
0.
\]
Together with the Fourier coefficient identity
\eqref{eq:rI-fouriercoeff}, we therefore obtain
\begin{align}
\E\left[\int_{T_I}
\bigl(r_{u,I,D}^{(<K)}\bigr)^2\,dQ_I\right]
&\le
\E\left[\int
\bigl(r_{u,I,D}^{(<K)}\bigr)^2\,dQ_I\right] =
\sum_{1\le |S|<K}
\E\left[\left|
\E[\eta_{D,S}\mid\bG_{-u}]
\right|^2\right].
\label{eq:weighted-fourier-low}
\end{align}
It remains to control the high-frequency contribution to \eqref{eq:weighted-fourier-split}. Identify $I\equiv[m]$, and define
\[r \equiv 
r_I:=
\min\Big\{|I|,\lceil 2|I|p+100\log n\rceil\Big\}\,.
\]
  Then the set $T_I$, as defined by \eqref{eq:local-row-truncation}, is the same as the set $T_{m,r}$ defined by \eqref{e:def.set.T.m.r}.   We then apply Proposition~\ref{p:high-Fourier-tail} with
$z=x_u$,
$X_V=x_{V\setminus u}$, 
  $X_{[m]}=x_I$,  
and $\mathcal F=\sigma(\bG_{-u})$. Note that $\bG_{-u}$ 
is measurable with respect to the
latent configuration excluding $x_u$, so
    \[\mathcal F
    \subseteq \sigma(x_{V\setminus u},x_{[m]})
    \]
  as required by Proposition~\ref{p:high-Fourier-tail}.   Recalling again
\eqref{eq:rI-fouriercoeff}, Proposition~\ref{p:high-Fourier-tail} gives
\begin{align}
\E\bigg[\int_{T_I}
\bigl(r_{u,I,D}^{(\ge K)}\bigr)^2\,dQ_I\bigg]
\le
\E(D^2)\,B_{m,K,r}\,.
\label{eq:weighted-fourier-high}
\end{align}
Since $r_I$ satisfies the hypothesis of
Proposition~\ref{p:bounding-B},
we have $
B_{m,K,r}=o(1/n)$
for every fixed $K>2/\delta$.
Consequently,
\[
\E(D^2)\,B_{m,K,r}
=
o\bigg(\frac{\E(D^2)}{n}\bigg)
=
o\bigg(\E(D^2)pL^{\nu(D)}\bigg)
=o(g(D))\,,\]
where in the second equality we used
$npL^{\nu(D)}\to\infty$, which follows from
$np\ge n^\delta$ and $L\ge 1$ for all sufficiently large $n$.
Substituting \eqref{eq:weighted-fourier-low} and
\eqref{eq:weighted-fourier-high} into
\eqref{eq:weighted-fourier-split} gives
\[
\E\bigg[\int_{T_I\cap\{\Lambda_{u,I}\ge1/2\}}
\frac{r_{u,I,D}^2}{\Lambda_{u,I}}\,dQ_I\bigg]
\le
4\sum_{1\le |S|<K}
\E\Big[
\E(\eta_{D,S}\mid\bG_{-u})^2\Big]
+
o(g(D))\,,
\]
which concludes the proof.
\end{proof}

In light of Lemma~\ref{lem:weighted-fourier}, it remains to control $\mathbb{E}[\mathbb{E}(\eta_{D,S}\mid \mathbf{G}_{-u})^2]$ for $1\le |S|<K$. Moreover, by exchangeability, $\mathbb{E}\left[|\mathbb{E}[\eta_{D,S}\mid \mathbf{G}_{-u}]|^2\right]$ depends on $S$ only through its size $|S|$. Therefore, the problem reduces to controlling a bounded number of terms.


\subsection{Step 6: Harmonic expansion and its high-degree remainder}
\label{sec:cavity-harmonic}

To control the posterior mean of $\eta_{D,S}$, we further expand $\eta_{D,S}$ in harmonic polynomials. This expansion is central to our recursive self-bounding scheme, which relates harmonic polynomials of varying but uniformly bounded degrees. The resulting bounds involve only the corresponding posterior means and $\Delta$, leading to a system of self-bounding inequalities.

To prepare for this, we collect some harmonic polynomial preliminaries. 
Recall that $\calH_\ell\equiv \calH_{\ell,d}$ is the space of degree-$\ell$ spherical
harmonics on $\mathbb S^{d-1}$ and that
$N_\ell\equiv N_{\ell,d}=\dim(\calH_{\ell,d})$. We note for later use that (see \cite[Corollary~1.1.4]{dai2013approximation})
    \begin{align} \nonumber
    &N_\ell 
    = \binom{d+\ell-1}{\ell}
    -\binom{d+\ell-3}{\ell-2} 
      = \binom{d+\ell-2}{\ell}
    +\binom{d+\ell-3}{\ell-1} \\
    &\qquad=\frac{d+2\ell-2}{d-2}
    \binom{d+\ell-3}{\ell}
    \asymp_\ell d^\ell\,,
    \label{e:harmonic.dim.estimate}
    \end{align}
where the last estimate holds for any fixed $\ell$.   It is also straightforward to verify that $N_{\ell,d}$ is nondecreasing in $\ell$.   Let $P_\ell\equiv P_{\ell,d}$ be the normalized
zonal harmonic characterized by $P_\ell(1)=1$ and the reproducing identity (\cite[Lemma~1.2.5 and Corollary~1.2.7]{dai2013approximation})
\begin{equation}
 f(x)=N_{\ell}\E_z\left[
   P_{\ell}(\langle x,z\rangle)f(z)
 \right],\qquad f\in\calH_{\ell}\,.
 \label{eq:reproducing-identity}
\end{equation}
Recall that $Y_\ell\equiv (Y_{\ell,j} : j\in[N_\ell])$ is an orthonormal basis of $\calH_\ell$, so that we have
    \begin{equation}\label{e:y.ytranspose}
        \E[ Y_\ell(x)\otimes Y_\ell(x) ]
        = I_{\calH_\ell}\,.
    \end{equation}
By the addition formula (\cite[Equation~(1.2.8)]{dai2013approximation}), we have the relation
    \begin{equation}
    \label{e:addition.formula}
    \frac{\langle Y_\ell(x),Y_\ell(y)\rangle}{N_\ell}
    = P_\ell(\langle x,y\rangle)
    = \frac{\ell+\lambda}{\lambda   N_\ell}
        C_{\ell,\lambda}(\langle x,y\rangle)
    = \frac{C_{\ell,\lambda}(\langle x,y\rangle)}{C_{\ell,\lambda}(1)}
    \,,
    \end{equation}
where $C_{\ell,\lambda}$ are the Gegenbauer polynomials with $\lambda=(d-2)/2$.

 Write the zonal harmonic expansion of the centered cap kernel as
\begin{equation}
 \kappa(\langle x,z\rangle)
 =\sum_{\ell\ge1}\alpha_\ell
   P_\ell(\langle x,z\rangle)\,,
 \label{eq:kappa-harmonic-expansion}
\end{equation}
where $\alpha_\ell\equiv \alpha_{\ell,d}$. 
For a multidegree
$\bfa=(a_i: i \in S) \in\naturals^S$ and $h\geq 0$, set
\begin{equation}\label{e:def.q.s}
A\equiv A(\bfa)\equiv \|\bfa\|_1\,,\quad
 q\equiv q(h;\bfa) \equiv \frac{A+h}{2}\,,
 \quad
 s\equiv s(h;\bfa) 
 \equiv \frac{A-h}{2}
 \,.
\end{equation}
For $\alpha_\ell$ the coefficients in the expansion \eqref{eq:kappa-harmonic-expansion}, we abbreviate  
\begin{equation}
 \gamma_{h;\bfa}
 :=\frac{1}{d^q}
 \prod_{j\in S}\alpha_{a_j}\,,
 \label{eq:transition-coefficient-definition}
\end{equation}
with $q=q(h;\bfa)$ as defined above.
Define the kernel
\begin{align}
 \sfA_{h;\bfa}(x_u,x_S)
 &:=d^q
 \int
 P_h(\ip{x_u}{z})
 \prod_{j\in S} P_{a_j}(\ip{x_j}{z})\,d\sigma(z) 
 \label{eq:harmonic-contraction}
\end{align}
where the integration is over the uniform probability measure $\sigma$ on $\mathbb S^{d-1}$.
We will see below that this kernel appears naturally in the expansion of $\eta_{D,S}$. The following lemma gives conditions for the kernel to vanish:

\begin{lemma}[harmonic selection rule]
\label{lem:selection-rule}
For the kernel $\sfA_{h;\bfa}$ (defined by \eqref{eq:harmonic-contraction}) to be non-zero, it is necessary that
\begin{equation}
q(h;\bfa)= \frac{h+\|\bfa\|_1}{2}\in\mathbb Z\,,
 \qquad
 q(h;\bfa) \ge 
 \max\bigg\{h,\max_{1\le j\le k}a_j\bigg\}\,.
 \label{eq:selection-rule-full}
\end{equation}
Equivalently, the total degree is even, and none of the degrees
$h,a_1,\ldots,a_k$ exceeds the sum of all the others.  In particular, we must have
\begin{equation}
 A = \|\bfa\|_1\ge \max\{h,k\}\,,
 \qquad 
 s(h;\bfa)=\frac{\|\bfa\|_1-h}{2}\in \mathbb Z\,.
 \label{eq:selection-rule}
\end{equation}
\end{lemma}

\begin{proof}
As a function of the integration variable $z$, each factor
$P_{a_j,d}(\langle x_{i_j},z\rangle)$ is a spherical harmonic of degree $a_j$,
and $P_{h,d}(\langle x_u,z\rangle)$ has degree $h$.  Under $z\mapsto-z$, the
integrand in \eqref{eq:harmonic-contraction} is multiplied by
$(-1)^{h+\|\bfa\|_1}$.  Hence the integral vanishes unless the total degree $h+\|\bfa\|_1$ is even.

Fix one of the degrees, say $a_j$.  The product of all remaining factors is
the restriction to the sphere of a polynomial in $z$ of degree at most
$h+\|\bfa\|_1-a_j$.  Since degree-$a_j$ spherical harmonics are orthogonal to every
polynomial of degree strictly less than $a_j$, the integral vanishes whenever
$a_j>h+\|\bfa\|_1-a_j$.  Applying the same argument to the degree-$h$ factor shows
that the integral also vanishes when $h>\|\bfa\|_1$.  Thus every degree among
$h,a_1,\ldots,a_k$ is at most $(h+\|\bfa\|_1)/2$, proving
\eqref{eq:selection-rule-full}.  Finally, $a_j\ge1$ for every
$1\le j\le k$, so $k\le\|\bfa\|_1$; the remaining conclusions in
\eqref{eq:selection-rule} follow immediately.
\end{proof}


Let $D=D(x_V) \in \calH_{\bsell}$. As above, we fix a marked vertex $u\in V$ and denote its degree $h\equiv \ell_u$. Given $S\subseteq V^c$,
and a multidegree $\bfa=(a_i: i\in S)$, we define the operator
\begin{align}
 (\sfT_{\bfa}^{u,S} D)(x_{V\setminus u},x_S)
 &:= N_h \E_{x_u}\Big[
 D(x_u,x_{V\setminus u})
 \sfA_{h;\bfa}(x_u,x_S)
 \Big]\,.
 \label{eq:generated-state-definition}
\end{align}
Acting on $D(x_V)=D(x_u,x_{V\setminus u})$, the operator $\mathsf T_{\bfa}^{u,S}$ produces a new polynomial that no longer depends on the marked variable $x_u$, but depends on new variables $x_S$.

We now explain the role of the kernel \eqref{eq:harmonic-contraction} and operator \eqref{eq:generated-state-definition} in the expansion of $\eta_{D,S}$ (as defined by \eqref{eq:eta-DS}):
	Substituting the univariate expansion \eqref{eq:kappa-harmonic-expansion} for each $\kappa$ in
	\prettyref{eq:eta-DS}, we have
		\begin{align}
	\eta_{D,S}
	&=
    \sum_{\bfa\in\naturals^S} \bigg(
	\prod_{j\in S} \alpha_{a_j} \bigg)
	\E_{x_u}\bigg[
	D(x_u,x_{V\setminus u})
	\prod_{j\in S} 
	P_{a_j}(\langle x_u,x_j \rangle)
	\bigg] \nonumber  \\ 
	&= \sum_{\bfa\in\naturals^S} 
	\gamma_{h;\bfa}
	\underbrace{N_h
	\E_{x_u}\bigg[
	D(x_u,x_{V\setminus u})
	\sfA_{h;\bfa}(x_u,x_S)
	\bigg]}_{(\mathsf{T}^{u,S}_{\bfa} D)
    (x_{V\setminus u},x_S)}
     \label{eq:eta_D_S_expression}
	\end{align}
	where the last step recalls
\eqref{eq:reproducing-identity}, \eqref{eq:transition-coefficient-definition},
    \eqref{eq:harmonic-contraction} and 
		\prettyref{eq:generated-state-definition}.
Under the product uniform measure on $x_S$,
\eqref{eq:eta_D_S_expression} is the \textbf{orthogonal $L^2$ expansion of $\eta_{D,S}$ by
multidegree}, with $x_{V\setminus u}$ treated as fixed.
Indeed, 
$\sfA_{h;\bfa}(x_u,x_S)$ and hence $\mathsf{T}^{u,S}_{\bfa} D(x_{V\setminus u},x_S)$ as a function of $x_S$ belong to $\calH_{\bfa}$.
Because $\calH_{\bfa}$'s are orthogonal for different multidegrees, 
the projection of $\eta_{D,S}$ onto $\calH_{\bfa}$ is precisely 
$\gamma_{h;\bfa}\mathsf T_{\bfa}^{u,S}D$.

A desirable property of the averaging operator \prettyref{eq:generated-state-definition} is that it preserves the tensor-product structure: If $D \in \calH_{\bsell}$, 
we have $\sfT_{\bfa}^{u,S} D \in \calH_{\bsell'}$, where the new multidegree $\bsell'$ is obtained by deleting the degree of the marked vertex $u$ and adding those of the new variables in $S$ (see \prettyref{lmm:Tproperty} in \prettyref{sec:Tproperty}).
This property allows the cavity recurrence
in~\prettyref{sec:cavity-recurrence-final} to be formulated within the same
family of tensor-product spherical harmonic polynomials. Together with remainder estimates below, it allows us to close the recursion over polynomials 
of bounded total degree, which we will do in~\prettyref{sec:closing-recursion}.

Generalizing \prettyref{thm:harmonic-expansion} (corresponding to $D\equiv1$), the following result describes the harmonic expansion of $\eta_{D,S}$ using the averaging operator $T_{\bfa}^{u,S}$
and provides the needed remainder estimates.

\begin{lemma}[harmonic expansion of $\eta_{D,S}$]
\label{lem:state-harmonic-expansion}
Let $D\equiv D(x_V)\in\calH_{\bsell}$ be nonzero, let $u\in V$, and set
$h=\ell_u$. Fix integers $k,M\ge1$.
For any $S\subset V^c$ with $|S|=k$, the quantity $\eta_{D,S}$ defined by
\eqref{eq:eta-DS} has the orthogonal truncation
\begin{equation}
 \eta_{D,S}
 =\sum_{\bfa\in\calA_{k,M}(h)}
    \gamma_{h;\bfa}(\mathsf T_{\bfa}^{u,S} D)(x_{V\setminus\{u\}},x_S)
    +\mathcal R_{D,S,M},
 \label{eq:eta-DS-expansion}
\end{equation}
where, recalling \eqref{e:def.q.s} and Lemma~\ref{lem:selection-rule}, we define
    \begin{equation}
    \label{e:def.calA.h}
    \calA_{k,M}(h)
    := \bigg\{
    \bfa \in\naturals^k
    : 
    \textup{
    $q(h;\bfa) \le M$
    and 
    $\bfa$ satisfies \eqref{eq:selection-rule}}
    \bigg\}\,.
    \end{equation}
There are constants $C_0,C>0$, depending only on $h,k,M$, such that
if $d\ge 4$ and $0<p\le e^{-1}$, then for any
$\bfa\in\calA_{k,M}(h)$,
\begin{equation}
 (\gamma_{h;\bfa})^2
 \le  C\frac{p^kL^{\|\bfa\|_1}}{d^h},
 \qquad
 \E[(
 T_{\bfa}^{u,S} D)^2]
 \le C \frac{\E(D^2)}{d^s}\,,
 \label{eq:transition-scales}
\end{equation} 
where $s = (\|\bfa\|_1-h)/2$. If, in addition, $d\ge C_0L^2$,
the harmonic remainder satisfies
\begin{equation}
 \E[(\mathcal R_{D,S,M})^2]
 \le
 C \E(D^2)\frac{p^kL^{2M+2-h}}{d^{M+1}}\,,
 \label{eq:harmonic-remainder-quantitative}
\end{equation}
where the expectations are under the independent uniform prior on the
arguments of each integrand. 
\end{lemma}

The \hyperlink{proof:l.roadmap-state-harmonic-expansion}{proof of Lemma~\ref{lem:state-harmonic-expansion}} is deferred to Section~\ref{sec:high-harmonic-remainder}. For now, we comment only that equation~\prettyref{eq:eta-DS-expansion} follows directly by truncating
the orthogonal expansion in~\prettyref{eq:eta_D_S_expression}; 
the main term in \prettyref{eq:eta-DS-expansion}
is the best $L^2$ approximation in the
variables $x_S$ from the sum of harmonic subspaces with total degree
at most $2M-h$.
The reason we truncate the expansion according to $q(h;\bfa) \le M$ is that 
\prettyref{eq:transition-scales} yields an upper bound on $|\gamma_{h;\bfa}|^2  \E[(
 T_{\bfa}^{u,S} D)^2]$ proportional to $d^{-{q(h;\bfa)}}$, so that the total residual contribution \prettyref{eq:harmonic-remainder-quantitative} decays as $d^{-(M+1)}$.

\subsection{The resulting cavity recurrence}
\label{sec:cavity-recurrence-final}

We can now state the cavity recurrence for posterior moments of spherical harmonics, using the preceding six steps.  For brevity, collect the 
multi-indices across the Fourier levels by writing
\begin{equation}\label{eq:calA.star.K.M.h}
 \calA^*_{K,M}(h)
 =\bigcup_{k=1}^{K-1}\calA_{k,M}(h)\,,
\end{equation}
for $\calA_{k,M}(h)$ as defined by \eqref{e:def.calA.h}.
Choose the fixed cutoffs $K,M$ so that
\begin{equation}
 M\ge K,\qquad
 K>2/\delta,\qquad
 \delta(3M+4-k)>1\quad(1\le k<K),\qquad
 \delta(M+4)>1.
 \label{eq:cutoff-conditions}
\end{equation}
The first condition above chooses the harmonic cutoff $M$ after the Fourier
cutoff $K$, the second will make the Fourier remainder above level $K$ negligible,
the third will make the high-degree harmonic remainder negligible after summing
over all $S$ with $1\le |S|<K$, and the fourth condition will be used later when closing the recursion (see~\eqref{eq:terminal-bound}).

\begin{proposition}[cavity recurrence]
\label{prop:cavity}
Let $K$ and $M$ satisfy~\prettyref{eq:cutoff-conditions}. For every  $D=D(x_V) \in \calH_{\bsell}$, fix $u\in V$ and let $h\equiv \ell_u$. Assume $|V| \le B_*$, where the slack $B_*$ is defined in \eqref{eq:local-row-truncation_slack}. We then have
\begin{align}
 \Post(D)
 &\le Cg(D)
   +C \E(D^2) \phi_{\nu(D)}(\Delta) \notag\\
 &\quad+C 
 \sum_{\bfa\in\calA^*_{K,M}(h)}
    \sum_{S \subseteq V^c: |S|=|\bfa|} \frac{p^{|\bfa|}L^{\|\bfa\|_1}}{d^h}
    \Post(\mathsf T_{\bfa}^{u,S} D)
   +o(g(D))
 \label{eq:cavity-schematic}
\end{align}
where $|\bfa|$ denotes the length of $\bfa$; and the constant $C$ only depends on $|V|$, $t(D)$, $K$, and $M$.
\end{proposition}
\begin{proof}
Recall $I=V^c$. By
Lemma~\ref{lem:restoration} and Lemma~\ref{lem:puncturing},
\begin{equation}
 \Post(D)
 \le Cg(D)
   +4\E\bigg[
   \Big(\E[D\mid\bG_{-u},\bY_I]\Big)^2
   \bigg]
= Cg(D) + 4\E\bigg[\int\frac{r_{u,I,D}^2}{\Lambda_{u,I}}\,dQ_I\bigg]\,.
 \label{eq:cavity-proof-restoration}
\end{equation}
We decompose this integral over $(T_I)^c$,
$T_I\cap\{\Lambda_{u,I}<1/2\}$, and
$T_I\cap\{\Lambda_{u,I}\ge1/2\}$, for $T_I$ as defined by
\eqref{eq:local-row-truncation}.  Lemmas~\ref{lem:discarded-row}
and~\ref{lem:denominator} bound the
first two contributions by
\begin{equation}
 o(g(D))
 +C\E(D^2) \phi_{\nu(D)}(\Delta)\,.
 \label{eq:cavity-proof-first-pieces}
\end{equation}
It remains to control the third contribution.
The truncated Fourier expansion in Lemma~\ref{lem:weighted-fourier} gives
\begin{equation}
\E\left[\int_{T_I\cap\{\Lambda_{u,I}\ge1/2\}}
       \frac{r_{u,I,D}^2}{\Lambda_{u,I}}\,dQ_I\right]
 \le 4 \sum_{k=1}^{K-1}
   \sum_{\substack{S\subseteq I\\|S|=k}}
   \E\left[\left|\E[\eta_{D,S}\mid\bG_{-u}]\right|^2\right] + o(g(D)).
 \label{eq:cavity-proof-low-parseval}
\end{equation}
Fix $S\subseteq I$ with $|S|=k<K$. By assumption~\eqref{eq:regime}, we have
$d\gg L^2$, so $d\ge C_0(h,k,M)L^2$ for large enough $n$.  
Lemma~\ref{lem:state-harmonic-expansion} then gives
\[
 \eta_{D,S}
 =\sum_{\bfa\in\calA_{k,M}(h)}
   \gamma_{h;\bfa}
   (\mathsf T^{u,S}_{\bfa}D)(x_{V\setminus u},x_S)
   +\mathcal R_{D,S,M}\,,
\]
with $\calA_{k,M}(h)$ as in \eqref{eq:eta-DS-expansion}.
Put
\[
 N_*:=1+\max_{1\le k<K}|\calA_{k,M}(h)|.
\]
Since $q\le M$ implies $h+\|\bfa\|_1\le2M$, we have
$N_*\le C_{K,M}$. 
Conditional expectation followed by Cauchy--Schwarz yields
\begin{align}
 \left|\E[\eta_{D,S}\mid\bG_{-u}]\right|^2
 &\le N_*\sum_{\bfa\in\calA_{k,M}(h)}
   |\gamma_{h;\bfa}|^2
   \left|\E\left[(\mathsf T_{\bfa}^{u,S}D)(x_{V \setminus u},x_S)\mid\bG_{-u}\right]\right|^2
   \notag\\
 &\qquad+N_*\left|\E[\mathcal R_{D,S,M}\mid\bG_{-u}]\right|^2\,.
 \label{eq:cavity-proof-finite-CS}
\end{align}
Because $\sigma(\bG_{-u})\subseteq\sigma(G_n)$, Jensen's inequality implies that 
\begin{align}
 \E\left[\left|\E\left[(\mathsf T_{\bfa}^{u,S} D)(x_{V \setminus u},x_S)\mid\bG_{-u}\right]\right|^2\right]
 &\le \Post(\mathsf T_{\bfa}^{u,S} D),
 \label{eq:cavity-proof-projection}\\
 \E\left[\left|\E[\mathcal R_{D,S,M}\mid\bG_{-u}]\right|^2\right]
 &\le \E[\mathcal R_{D,S,M}^2].
 \label{eq:cavity-proof-remainder-contraction}
\end{align}
 Therefore, 
combining \eqref{eq:cavity-proof-finite-CS}--\eqref{eq:cavity-proof-remainder-contraction} with the coefficient bound
\eqref{eq:transition-scales}
gives
\begin{align}
 &\sum_{k=1}^{K-1}
   \sum_{\substack{S\subseteq I\\|S|=k}}
   \E\left[\left|\E[\eta_{D,S}\mid\bG_{-u}]\right|^2\right] \notag\\
 &\quad\le
 C\sum_{\bfa\in\calA^*_{K,M}(h)}
   \sum_{S \subset I: |S|=|\bfa|} \frac{p^{|\bfa|}L^{\|\bfa\|_1}}{d^h}
   \Post(\mathsf T_{\bfa}^{u,S} D)
 +\operatorname{Tail}_{\mathrm{har}}(D)\,,
 \label{eq:cavity-proof-transition-sum}\\
 \operatorname{Tail}_{\mathrm{har}}(D)
 &:=C\sum_{k=1}^{K-1}
   \sum_{\substack{S\subseteq I\\|S|=k}}
   \E[(\mathcal R_{D,S,M})^2].
 \label{eq:cavity-proof-harmonic-tail}
\end{align}
By assumption~\eqref{eq:regime}, we have $d\gg L^2$.
Recalling from
\eqref{e:def.L.eps} that
$\varepsilon_n=(npL)^3/d$, 
the estimate
\eqref{eq:harmonic-remainder-quantitative}
from Lemma~\ref{lem:state-harmonic-expansion} gives
\begin{align}
\frac{\operatorname{Tail}_{\mathrm{har}}(D)}
      {g(D)}
 &\le C\sum_{k=1}^{K-1}
  \frac{n\varepsilon_n^{M+1}}
       {(np)^{3M+4-k}L^{M+1+h+\nu(D)}}
 =o(1).
 \label{eq:cavity-proof-harmonic-tail-small}
\end{align}
Indeed, the terms vanish by
\eqref{eq:cutoff-conditions}, since $np\ge n^\delta$ and
$\varepsilon_n=o(1)$.
Substitute \eqref{eq:cavity-proof-transition-sum} and \eqref{eq:cavity-proof-harmonic-tail-small} into
\eqref{eq:cavity-proof-low-parseval}, then use
\eqref{eq:cavity-proof-restoration}--\eqref{eq:cavity-proof-first-pieces}.
All constants arising from Cauchy--Schwarz
and the fixed cutoffs are absorbed into $C$.  This gives exactly
\eqref{eq:cavity-schematic} and proves the proposition.
\end{proof}

\section{Closing the posterior recursion}
\label{sec:closing-recursion}

In this section we complete the \hyperlink{proof:t.posterior-components}{proof of Theorem~\ref{thm:posterior-components}}.  The
one-step recurrence has already been derived in
Proposition~\ref{prop:cavity}. We
close it over a
bounded-degree class: with $\calH_{\bsell}$ as in \eqref{eq:tensor-harmonics}, define 
\begin{equation}
 \mathfrak H_{\le2M}
 :=\bigcup_{t=1}^{2M} \bigcup_{V \subseteq[n]} \ \bigcup_{\substack{\bsell: V \to \naturals\\
                       \| \bsell\|_1=2t\\
                     }}
   \bigl(\calH_{\bsell}\setminus\{0\}\bigr)\,.
 \label{eq:bounded-harmonic-class}
\end{equation}
Equivalently, recalling that $t(D)=\|\bsell\|_1/2$,
this is the union of the positive multidegree spaces with
$t(D)\le 2M$.

Recall the definition of $g(D)$ from \eqref{eq:gD}.
For $D\in\mathfrak H_{\le 2M}$, define
\begin{align}\label{eq:XD}
 X(D) &:=
 \frac{\Post(D)}
 {g(D)(np)^{(t(D)-M)_+}}\,,\\
Z & :=\sup_{D\in\mathfrak H_{\le 2M}}X(D) .
 \label{eq:Z}
\end{align}
First observe $Z<\infty$: indeed, the conditional Jensen's inequality gives
$\Post(D)\le\E(D^2)$, so
\begin{align}
Z=\sup_{D\in\mathfrak H_{\le 2M}}X(D)
\le \sup_{D\in\mathfrak H_{\le 2M}}
 \frac{1}{pL^{\nu(D)}(np)^{(t(D)-M)_+}} 
\le\frac1p\,.
 \label{eq:Z_trivial_bound}
\end{align}
The next proposition gives a much sharper bound on $Z$, which in turn implies Theorem~\ref{thm:posterior-components}.

\begin{proposition}[closing the posterior recursion]
\label{prop:closure}
  Assume the setting of Theorem~\ref{thm:main-detection}, and take cutoffs $K,M$ satisfying \eqref{eq:cutoff-conditions}.   Choose $B_*\ge 4M$. There exists deterministic $0\le \theta_n =o(1)$ such that the quantity $Z$ defined by \eqref{eq:Z} satisfies the bound
\begin{equation}
 Z\le C\bigg(1+ \frac{\Delta}{p}\bigg)
 + \theta_n 
 Z+o(1)\,,
 \label{eq:closed-envelope}
\end{equation}
  for a constant $C=C(\delta,K,M)$.  Adjusting the constant appropriately gives $Z\le C(1+\Delta/p)$.
\end{proposition}

\begin{proof}
Apply Proposition~\ref{prop:cavity} to an arbitrary
$ D \in\mathfrak H_{\le 2M}$ and divide by the denominator in
\eqref{eq:XD}. 
Recall that Proposition~\ref{prop:cavity} relies on an application of Lemma~\ref{lem:state-harmonic-expansion}: since $h\le2t(D)\le4M$ on this class and $1\le k<K$,
the constants $C_0(h,k,M)$ in
Lemma~\ref{lem:state-harmonic-expansion} have a common upper bound
depending only on $K,M$. By assumption~\eqref{eq:regime}, we have
$d\gg L^2$, so $d\ge C_0(h,k,M)L^2$ holds for all large enough $n$,
uniformly over the class $\mathfrak H_{\le 2M}$. 
  Assume $D\equiv D(x_u, x_{V \setminus u}) \in \calH_{\bf \ell}$ and let $h=\ell_u$. For a multidegree $\mathbf a$, write $k=|\bfa|$,   and recall the notation from \eqref{e:def.q.s}.  
For every $S \subseteq V^c$ with $|S|=k$, if $\mathsf T^{u,S}_{\mathbf a}D\in
\mathfrak H_{\le2M}$, define
\begin{align}
 \beta^{u,S}_{\mathbf a}(D)
 &:=\frac{p^kL^{ A}}{d^h}
    \frac{g(\mathsf T^{u,S}_{\mathbf a}D)
    (np)^{(t(\mathsf T_{\mathbf a}^{u,S}D)-M)_+}}
    {g(D)
    (np)^{(t(D)-M)_+}}\,
    .
 \label{eq:beta-definition}
\end{align}
For every $\mathsf T^{u,S}_{\mathbf a}D$ with
$t(\mathsf T^{u,S}_{\mathbf a}D)>2M,$ define
\begin{equation}
 r_{\mathbf a}^{u,S}(D)
 := 
 \frac{p^{k}L^{ A}}{d^h}
 \frac{\Post(\mathsf T^{u,S}_{\mathbf a}D)}
    {g(D)(np)^{(t(D)-M)_+}}\,.
 \label{eq:terminal-definition}
\end{equation}
The divided cavity recurrence~\eqref{eq:cavity-schematic} then simplifies to
\begin{align} \nonumber
 X(D)
 &\le C
 +\frac{C\phi_{\nu(D)}(\Delta)}
       {pL^{\nu(D)}(np)^{(t(D)-M)_+}} 
 + C
 \sum_{\bfa\in\calA^*_{K,M}(h)}
 \sum_{\substack{S \subseteq V^c: |S|=|\bfa|\\t(\mathsf T^{u,S}_{\mathbf a}D) \le 2M }}
 \beta_{\mathbf a}^{u,S}(D)X(\mathsf T^{u,S}_{\mathbf a}D) \\
 &\qquad +C
 \sum_{\mathbf a\in\calA^*_{K,M}(h)}
\sum_{\substack{S \subseteq V^c: |S|=|\bfa|\\t(\mathsf T^{u,S}_{\mathbf a}D)>2M }}
 r^{u,S}_{\mathbf a}(D)+o(1)\,,
 \label{eq:X-recursion}
\end{align}
where the constant $C=C_{K,M}$ only depends on $K$ and $M$, and the sums omit terms for which $\mathsf T^{u,S}_{\mathbf a}D = 0$.

  For the second term on the right-hand side of \eqref{eq:X-recursion}, recall the definition \eqref{eq:def-nu-D} of $\nu(D)$. In particular, the restriction $t(D)\le 2M$ implies $\nu(D) \le 4M$.   We claim
that there exists a constant $C'$, depending only on $K,M,\delta$, such that  
\[
\frac{\phi_{\nu(D)}(\Delta)}
       {pL^{\nu(D)}(np)^{(t(D)-M)_+}}\le C'p^{-1}\Delta+o(1)
\]
holds uniformly over $D\in \mathfrak{H}_{\le 2M}$.
Recall the definition of $\phi_\nu$ from \eqref{e:def.phi.nu}.
If $\Delta\ge p^2$, then we have
    \[\frac{\phi_{\nu(D)}(\Delta)}{pL^{\nu(D)}(np)^{(t(D)-M)_+}}
    \le \frac{\phi_{\nu(D)}(\Delta)}{p L^{\nu(D)}}
    \lesssim \frac{\Delta}{p}\,,
    \]
where the hidden constant is uniform over $ \nu(D)\le 4M$. On the other hand, 
for all sufficiently large $n$,
the function $\phi_{\nu(D)}$ is nondecreasing on $[0,p^2]$, so for $\Delta\le p^2$ we can bound
\[
\frac{\phi_{\nu(D)}(\Delta)}{pL^{\nu(D)}(np)^{(t(D)-M)_+}}\le \frac{p^2(1+2L)^{\nu(D)}}{pL^{\nu(D)}}\lesssim p=o(1)\,,
\]
again uniformly over $\nu(D)\le4M$. This verifies the claim.

Given the claim, we further define
\begin{align*}
 B_n &:=
 \sup_{D\in\mathfrak H_{\le2M}}
 \sum_{\mathbf a\in\calA^*_{K,M}(h)}
 \sum_{\substack{S \subset V^c: |S|=|\bfa|\\t(\mathsf T^{u,S}_{\mathbf a}D) \le 2M }}
 \beta^{u,S}_{\mathbf a}(D)\,,\\
 R_n &:=\sup_{D\in\mathfrak H_{\le2M}}  \sum_{\mathbf a\in\calA^*_{K,M}(h)}
 \sum_{\substack{S \subset V^c: |S|=|\bfa|\\t(\mathsf T^{u,S}_{\mathbf a}D) > 2M }} r^{u,S}_{\mathbf a}(D).
\end{align*}
Taking the supremum in \eqref{eq:X-recursion} and adjusting $C$ if necessary yield the preliminary
scalar recursion
\begin{equation}
 Z\le C\bigg( 1+ \frac{\Delta}{p}+B_nZ+R_n \bigg)+o(1)\,.
 \label{eq:preliminary-Z-recursion}
\end{equation}
It remains to prove that $B_n=o(1)$, $R_n=o(1)$.

  We begin with the analysis of $B_n$.   The
selection rule in Lemma~\ref{lem:selection-rule} ensures that $s$ (as defined by \eqref{e:def.q.s}) is a nonnegative integer. 
  The multidegree of $T_{\bfa}^{u,S} D$ is characterized by Lemma~\ref{lmm:Tproperty}, and it follows that  
\begin{equation}
2t(T_{\bfa}^{u,S} D) = 2t(D) - h + \|\bfa\|_1
=2(t(D) + s)\,.
\label{eq:degree-transition}
\end{equation}
Again recalling the definition $\nu(D)$ from \eqref{eq:def-nu-D}, we now claim that \begin{equation}
 \nu(\mathsf T_{\mathbf a}^{u,S} D)-\nu(D)\le 2s+1,
 \label{eq:nu-increment}
\end{equation}
Indeed, if neither function is a scalar multiple of $\iprod{x_u}{x_v}$ according to notation in \eqref{eq:def-nu-D}, then the left-hand side of
\eqref{eq:nu-increment} equals $2t(\mathsf T_{\mathbf a}^{u,S}D)-2t(D)=2s$ in view of~\eqref{eq:degree-transition}.  If exactly one belongs
to this one-dimensional subspace, the definition of $\nu$ gives \eqref{eq:nu-increment}; if both do, then \eqref{eq:nu-increment} holds trivially since the left-hand side is zero.
Combining \prettyref{eq:nu-increment} with
the contraction bound
\eqref{eq:transition-scales} 
from \prettyref{lem:state-harmonic-expansion}
gives 
\begin{equation}
 \frac{g(\mathsf T^{u,S}_{\mathbf a}D)}{g(D)}
 \stackrel{\eqref{eq:gD}}{=}
 \frac{
 \E[(\mathsf T^{u,S}_{\mathbf a}D)^2]
 L^{\nu(\mathsf T^{u,S}_{\mathbf a}D)}
 }{\E(D^2)
 L^{\nu(D)}}
 \stackrel{\eqref{eq:transition-scales} }{\le} C_{K,M}\frac{L^{2s+1}}{d^s}\,.
 \label{eq:g-ratio}
\end{equation}
Moreover, since the map $x\mapsto(x-M)_+$ is nondecreasing and $1$-Lipschitz, we have
\[
 0\le
 (t(\mathsf T^{u,S}_{\mathbf a}D)-M)_+-(t(D)-M)_+
 \le s.
\]
Again recalling \eqref{e:def.q.s}, substituting the above two displayed inequalities
directly into \eqref{eq:beta-definition}, and then summing over the  at most $n^k$ choices of $S$, gives
    \begin{align}\nonumber
 \beta_{\mathbf a}(D)
 &:= \sum_{\substack{S \subset V^c: |S|=|\bfa|\\t(\mathsf T^{u,S}_{\mathbf a}D) \le 2M }}
  \beta^{u,S}_{\mathbf a}(D) 
   \le n^k
  \cdot 
  \frac{p^k L^A}{d^h} \cdot 
  C_{K,M} \frac{L^{2s+1}}{d^s} \cdot 
  (np)^s \\
  & 
  = C_{K,M} 
  \bigg(\frac{(npL)^3}{d}\bigg)^q
  \frac{L^{q-3h+1}}{ (np)^{3q-k-s}}
  \stackrel{\eqref{e:def.L.eps}}{=}
   C_{K,M}
 \frac{(\varepsilon_n)^qL^{q-3h+1}}{(np)^{3q-k-s}}\,.
 \label{eq:beta-explicit}
\end{align}
Recalling Lemma~\ref{lem:selection-rule},
the selection rule $k\le { A}=2q-h$ gives
\begin{equation}\label{e:threeqminusk}
 3q-k\ge q+h=s+2h\ge s+2\,,
\end{equation}
 because $h\ge1$   by the assumption $D\in\mathfrak H_{\le 2M}$ (recall from \eqref{eq:bounded-harmonic-class} that we require $\ell_v\in\naturals$ for all $v$, including $\ell_u=h$).   In addition,   since $\bfa\in\calA^*_{K,M}(h)$ as defined by \eqref{eq:calA.star.K.M.h} and \eqref{e:def.calA.h}, we must have $q\le M$, and therefore  
 \[
  q-3h+1\le q-2\le M-2 \le M\,,
 \]
 Since $\varepsilon_n=o(1)$ and $L\ge1$, it follows that
 \begin{equation}
  \beta_{\mathbf a}(D)
 \le C_{K,M}\frac{L^M}{(np)^2} \,. 
  \label{eq:beta-bound}
 \end{equation}
Let $N_{K,M}: =\max_{1\le h\le 4M}|\calA^*_{K,M}(h)|$.
Then, substituting \eqref{eq:beta-bound}
into the above definition of $B_n$ gives
\begin{equation}
  B_n
 \le N_{K,M}C_{K,M}\frac{L^M}{(np)^2}
 \le C_{K,M}\frac{L^M}{(np)^2}
 =:\theta_n=o(1)\,,
 \label{eq:an}
\end{equation}
where the last estimate follows from $np\ge n^\delta$ and $L\le\log n$.

  We now turn to the analysis of $R_n$.   If $t(\mathsf T^{u,S}_{\mathbf a}D)>2M$, by the conditional Jensen's inequality and the contraction bound \eqref{eq:transition-scales} from Lemma~\ref{lem:state-harmonic-expansion}, we have
\[
 \Post(\mathsf T^{u,S}_{\mathbf a}D)
 \le \E[(\mathsf T^{u,S}_{\mathbf a}D)^2]
 \le \frac{C}{d^s}\E(D^2)\,.
\]
Substituting this estimate into \eqref{eq:terminal-definition}, and again recalling \eqref{eq:gD} and \eqref{e:def.q.s}, we obtain
\begin{align}
 r_{\mathbf a}^{u,S}(D)
 &\le
 \frac{1}{pL^{\nu(D)}
                   (np)^{(t(D)-M)_+}} \cdot
 \frac{p^kL^{ A}}{d^h} \cdot \frac{C_{K,M}}{d^s}
    =\frac{C_{K,M}p^kL^{2q-h-\nu(D)}}
         {d^q p(np)^{(t(D)-M)_+}}\,.
 \label{eq:terminal-summand-first}
\end{align}
Summing over the at most $n^k$ choices of $S$ gives
\begin{align}
 r_{\mathbf a}(D) &:=
 \sum_{\substack{S \subset V^c: |S|=|\bfa|\\t(\mathsf T^{u,S}_{\mathbf a}D) > 2M }} r^{u,S}_{\mathbf a}(D)
   \le
 \frac{C_{K,M} (npL)^{3q} L^{-q-h-\nu(D)}}
    {d^q p (np)^{3q-k + (t(D)-M)_+} }
   \stackrel{\eqref{e:def.L.eps}}{=}
 \frac{C_{K,M} (\varepsilon_n)^q
       L^{-q-h-\nu(D)}}
      {p(np)^{3q-k+(t(D)-M)_+}} \\ 
      &\stackrel{\eqref{e:threeqminusk}}{\le}
 \frac{C_{K,M}}
      {p(np)^{s+2+(t(D)-M)_+}}
      \stackrel{\eqref{eq:degree-transition}}{\le} \frac{C_{K,M}}
      {p(np)^{M+3}}\,.
 \label{eq:terminal-summand}
\end{align}
Hence, summing
\eqref{eq:terminal-summand} over
$\calA^*_{K,M}$ gives 
\begin{equation}
 \sum_{\mathbf a\in\calA^*_{K,M}(h)} r_{\mathbf a}(D)
 \le \frac{N_{K,M} C_{K,M}}{p(np)^{M+3}}
 =\frac{C_{K,M}n}{(np)^{M+4}}
 =o(1)\,,
 \label{eq:terminal-bound}
\end{equation}
  using the last assumption in \eqref{eq:cutoff-conditions}.   This proves $R_n=o(1)$.

Plugging these estimates into \eqref{eq:preliminary-Z-recursion}, we obtain that
\[
Z\le C\bigg(1+\frac{\Delta}{p}+ {  \theta_n}Z+o(1)\bigg)+o(1)\,,
\]
  which implies \eqref{eq:closed-envelope}. Since we noted above that $Z<\infty$, this in turn implies  
\[
Z\le C\bigg(1+\frac{\Delta}{p}\bigg) 
\]
for some adjusted constant $C$. This concludes the proof.
\end{proof}


With~\prettyref{prop:closure}, we are now ready to complete the proof of Theorem~\ref{thm:posterior-components}.
\begin{proof}[\hypertarget{proof:t.posterior-components}{Proof of Theorem~\ref{thm:posterior-components}}] Fix a component $D_{\bfa}$ in the expansion \eqref{eq:harmonic-expansion} from Theorem~\ref{thm:harmonic-expansion}. If $D_{\bfa}=0$, then the desired bound is immediate, so assume
$D_{\bfa}\ne0$. By Theorem~\ref{thm:harmonic-expansion},
$D_{\bfa}\in\calH_{\bfa}$, and
    \[
    q=\frac{\|\bfa\|_1}{2}=t(D_{\bfa}) \le M\,,
    \]
so $D_{\bfa}$ lies in the set $\mathfrak H_{\le2M}$ defined by
\eqref{eq:bounded-harmonic-class}.
In particular, the factor $(np)^{(t(D_{\bfa})-M)_+}$ in
\eqref{eq:XD} equals $1$, so we have simply
$\Post(D_{\bfa}) =X(D_{\bfa})g(D_{\bfa})$. We now compare the factor $g(D_{\bfa})$, which we recall is defined by \eqref{eq:gD}, with the factor $g_t$ defined by \eqref{eq:gt}. We have two cases:
\begin{itemize}
    \item If $q=1$, it falls into the case
    \eqref{eq:critical-component} of Theorem~\ref{thm:harmonic-expansion}, which
    shows that
$\bfa=(1,1)$ and
$D_{\bfa}(x_1,x_2)=\iprod{x_1}{x_2}$. Thus
$\nu(D_{\bfa})=1=\nu_1$ by \eqref{eq:gt} and \eqref{eq:def-nu-D}, and
$\E[(D_{\bfa})^2]=1/d$. Therefore
\[ g(D_{\bfa})
 \stackrel{\eqref{eq:gD}}{=}\E[(D_{\bfa})^2]\,pL^{\nu(D_{\bfa})}
  =
  \frac{pL}{d}
  \stackrel{\eqref{eq:gt}}{=}g_1\,.
\]
    \item Otherwise, if $q\ge2$, then $\nu(D_{\bfa})=2q=\nu_q$ by \eqref{eq:gt} and \eqref{eq:def-nu-D}. The contraction bound in
\eqref{eq:harmonic-component-scales} from Theorem~\ref{thm:harmonic-expansion} yields
\[g(D_{\bfa})
 \stackrel{\eqref{eq:gD}}{=}
  \E[(D_{\bfa})^2]\,pL^{\nu(D_{\bfa})}
  \stackrel{\eqref{eq:harmonic-component-scales}}{\le} C_{K,M}\frac{pL^{2q}}{d^q}
  \stackrel{\eqref{eq:gt}}{=}
  C_{K,M}g_q\,.
\]
\end{itemize}
In either case, combining with the result of 
Proposition~\ref{prop:closure} gives
\begin{equation}\label{eq:thm5-component-intermediate}
  \Post(D_{\bfa})
  =X(D_{\bfa})g(D_{\bfa})
  \le C_{\delta,K,M}g_q
  \bigg(1+\frac{\Delta}{p}\bigg)\,,
\end{equation}
thereby concluding the proof.
\end{proof}

\section{The high-frequency Fourier  remainder}\label{sec:high-level-fourier}

This section isolates the Fourier-tail estimates shared by Theorem~\ref{thm:Fourier-reduction} and   Lemma~\ref{lem:weighted-fourier}. 
  The main results of this section are Propositions \ref{p:high-Fourier-tail} and \ref{p:bounding-B} below.

We first introduce some notation.
For integers $0\le k\le m$ and $0\le r\le m$, define
\begin{equation}\label{e:weight-w}
w_{m,k,r}
:=
\mathbb{P}\Big(
\operatorname{Bin}(k,1-p)+\operatorname{Bin}(m-k,p)\le r
\Big)\,,
\end{equation}
where the two binomial random variables are independent. Then set
\begin{equation}\label{e:def.B.m.K.r}
B_{m,K,r}
:=
\sum_{k=K}^m
\binom{m}{k}
w_{m,k,r}\,
\mathbb{E}_{\rho}[|\eta_2(\rho)|^k]\,,
\end{equation}
where $\rho=\langle z,z'\rangle$ for independent $z,z'\sim \operatorname{Unif}(\mathbb{S}^{d-1})$, and we identify $\eta_2(z,z')$
(as defined by \eqref{eq:eta})
with $\eta_2(\langle z,z'\rangle)=\eta_2(\rho)$.  Define the set
\begin{equation}\label{e:def.set.T.m.r}
T_{m,r}:=\Big\{y\in \{0,1\}^m:\sum y_i\le r\Big\}\,.
\end{equation}
Recall from~\eqref{eq:eta-DS} that 
\begin{equation}\label{eq:eta-DS.review}
\eta_{D,S}=\eta_{D,S}(X_V,X_S)
:=
\mathbb{E}_z\bigg[
D(z,X_V)
\prod_{j\in S}
\kappa(\langle z,x_j\rangle)
\bigg]\,.
\end{equation}
Recall that $Q_m$ denotes the reference measure $\operatorname{Bern}(p)^{\otimes m}$ on $\{0,1\}^m$, and recall that $(\psi_S)_S$ is the corresponding Fourier basis defined by \eqref{eq:fourierbasis}.

\begin{proposition}[weighted high-frequency tail]\label{p:high-Fourier-tail} For $m\le n$, $h\ge 0$, and $V\subseteq [n]\setminus[m]$, let
$D(z,X_V)$ be a spherical harmonic
  (as in \eqref{eq:pure-block})   in the variables $z$ and
$X_V=\{x_v:v\in V\}$, with spherical-harmonic degree $h$ in $z$. When $h=0$, we allow $V=\emptyset$
and $D\equiv 1$. 
If $\mathcal F$ is any sub-$\sigma$-field of $\sigma(X_V,X_{[m]})$, then we have
\begin{equation}
\mathbb{E}_{\mathcal F}\bigg[
\int_{T_{m,r}}
\bigg(
\sum_{\substack{S\subseteq[m], |S|\ge K}}
\mathbb{E}(\eta_{D,S} \mid \mathcal F) \psi_S(y)
\bigg)^2
\,dQ_m(y)\bigg]
\le
\Expect(D^2) B_{m,K,r}
\label{e:high-frequency-tail}
\end{equation}
for $B_{m,K,r}$ as in \eqref{e:def.B.m.K.r},
$T_{m,r}$ as in \eqref{e:def.set.T.m.r}, and
$\eta_{D,S}$ as in \eqref{eq:eta-DS.review}.
\end{proposition}

\begin{proposition}[uniform scalar Fourier-tail bound]\label{p:bounding-B}
Fix two integers $K\ge 2$ and $B_*\ge 0$. There exists a constant $C$ such that for all sufficiently large $n$, uniformly over $m\le n$ and
\[
r
\le
\min\left\{
m,\,
\lceil 2mp+100\log n\rceil+B_*
\right\},
\]
we have
\[
B_{m,K,r}
\le
C
\left(
\left(
\frac{npL}{\sqrt{dK}}
\right)^K
+
\left(
\frac{npL^2}{d}
\right)^K
+
\exp\left(
CnpL-\frac{d}{CL^2}
\right)\right).
\]
In particular, since $np\ge n^{\delta}$ and $d\gg (npL)^3$, for any fixed $K>2/\delta$, we have $B_{m,K,r}=o(1/n)$. 
\end{proposition}

\begin{proof}[Proof of Proposition~\ref{p:high-Fourier-tail}]
By conditional Jensen, the left-hand side of \eqref{e:high-frequency-tail} is bounded by
\[
\mathbb{E}\left[
\int_{T_{m,r}}
\Bigg(
\sum_{\substack{S\subseteq[m]\\ |S|\ge K}}
\eta_{D,S} \psi_S(y)
\Bigg)^2
\,dQ_m(y)\right]=\sum_{\substack{S,T\subset [m]\\|S|,|T|\ge K}}\int_{T_{m,r}}\mathbb{E}\left[\eta_{D,S}\eta_{D,T}\right]\psi_S(y)\psi_T(y)dQ_m(y)\,.
\]
For $S\neq T\subseteq[m]$, taking expectation over any coordinate $x_i$ with $i\in S\triangle T$ leaves one centered cap factor and hence $\expect{\eta_{D,S}\eta_{D,T}}=0$. Thus, it remains to consider the diagonal terms $S=T.$ 

For $S=T$, since $X_{[m]},X_V$ are independent of $y$, the integrand factorizes as
\[
\int_{T_{m,r}}\psi_S(y)^2dQ_m(y)\times \mathbb{E}[\eta_{D,S}^2]\,.
\]
The first integral equals precisely $w_{m,|S|,r}$ defined in \eqref{e:weight-w}, since
under the tilted measure $\psi_S^2\,dQ$, the coordinates in
$S$ are independent $\operatorname{Bern}(1-p)$ variables, while the
coordinates outside $S$ remain independent $\operatorname{Bern}(p)$
variables. For the second factor, we claim that
\begin{equation}\label{e:second-moment-eta-DS}
\mathbb{E}_{X_V,X_{[m]}}[\eta_{D,S}^2]=\mathbb{E}[D^2]\mathbb{E}_{\rho}\left[P_{h,d}(\rho)\eta_2(\rho)^{|S|}\right]\,,
\end{equation}

Given this claim and the above bounds, the left-hand-side of \eqref{e:high-frequency-tail} is bounded by
\[
\sum_{\substack{S\subset [m]\\|S|\ge K}}w_{m,|S|,r}\times \mathbb{E}[D^2]\mathbb{E}_{\rho}\left[P_{h,d}(\rho)\eta_2(\rho)^{|S|}\right]=\mathbb{E}[D^2]\sum_{k=K}^m\binom{m}{k}w_{m,k,r}\mathbb{E}_{\rho}\left[P_{h,d}(\rho)\eta_2(\rho)^{k}\right]\,,
\]
which is in turn bounded by $\mathbb{E}[D^2]B_{m,K,r}$ since $|P_{h,d}(\rho)|\le 1$. This proves \eqref{e:high-frequency-tail}. 

We now verify the identity in \eqref{e:second-moment-eta-DS}.
Introduce an independent replica $z'$ of $z$. By definition,
\[
\eta_{D,S}^2
=
\mathbb{E}_{z,z'}
\left[
D(z,X_V)D(z',X_V)
\prod_{j\in S}
\kappa(\langle z,x_j\rangle)
\kappa(\langle z',x_j\rangle)
\right].
\]
Averaging first over $X_{[m]}$, and using the independence of the variables $x_j$, gives
\[
\mathbb{E}_{X_{[m]}}
\left[
\prod_{j\in S}
\kappa(\langle z,x_j\rangle)
\kappa(\langle z',x_j\rangle)
\right]
=
\eta_2(z,z')^{|S|}.
\]
By rotational invariance, $\eta_2(z,z')$ depends on $z,z'$ only through
$\rho=\langle z,z'\rangle$, so we write it as $\eta_2(\rho)$. Hence
\begin{equation}\label{e:special-identity}
\mathbb{E}_{X_V,X_{[m]}}\big[\eta_{D,S}^2\big]
=
\mathbb{E}_{z,z',X_V}
\left[
D(z,X_V)D(z',X_V)
\eta_2(\rho)^{|S|}
\right].
\end{equation}

It therefore remains to prove the following general identity: for any bounded function $f$,
\begin{equation}\label{e:general-identity}
\mathbb{E}_{z,z',X_V}
\left[
D(z,X_V)D(z',X_V)f(\langle z,z'\rangle)
\right]
=
\mathbb{E}[D^2]\,
\mathbb{E}_{\rho}
\left[
P_{h,d}(\rho)f(\rho)
\right].
\end{equation}
To see this, let $(Y_a)_{a=1}^{N_{h,d}}$ be an orthonormal basis of $\calH_{h,d}$. Since $z \mapsto D(z,X_V) \in \calH_h$, we may write
\begin{equation}
D(z,X_V)
=
\sum_{a=1}^{N_{h,d}}Y_a(z)F_a(X_V).
\label{eq:D-expandxu}
\end{equation}
Therefore the left-hand side of~\eqref{e:general-identity} equals
\begin{equation}\label{e:expansion}
\sum_{a,b=1}^{N_{h,d}}
\mathbb{E}[F_aF_b]\,
\mathbb{E}_{z,z'}
\left[
Y_a(z)Y_b(z')f(\langle z,z'\rangle)
\right].
\end{equation}
For fixed $z'$, the Funk--Hecke formula~\cite[Theorem 1.2.9]{dai2013approximation} gives
\[
\mathbb{E}_z
\left[
f(\langle z,z'\rangle)Y_a(z)
\right]
=
\lambda_h(f)Y_a(z'),
\qquad
\lambda_h(f)
=
\mathbb{E}_{\rho}
\left[
P_{h,d}(\rho)f(\rho)
\right].
\]
Consequently, by orthonormality,
\[
\mathbb{E}_{z,z'}
\left[
Y_a(z)Y_b(z')f(\langle z,z'\rangle)
\right]
=
\lambda_h(f)\,
\mathbb{E}_{z'}[Y_a(z')Y_b(z')]
=
\lambda_h(f)\mathbf{1}\{a=b\}.
\]
Substituting into~\eqref{e:expansion} yields
\[
\mathbb{E}_{z,z',X_V}
\left[
D(z,X_V)D(z',X_V)f(\langle z,z'\rangle)
\right]
=
\lambda_h(f)
\sum_{a=1}^{N_{h,d}}\mathbb{E}[F_a^2].
\]
Finally, orthonormality once again gives
\[
\sum_{a=1}^{N_{h,d}}\mathbb{E}[F_a^2]
=
\mathbb{E}_{z,X_V}[D(z,X_V)^2]
=
\mathbb{E}[D^2].
\]
This proves~\eqref{e:general-identity}. Taking
$
f(\rho)=\eta_2(\rho)^{|S|}
$
in \eqref{e:general-identity} and using~\eqref{e:special-identity} proves~\eqref{e:second-moment-eta-DS}.
\end{proof}

We now turn to the proof of Proposition~\ref{p:bounding-B}. 
The sum in $B_{m,K,r}$ is naturally encoded by a generating function. Define
\[
F_{m,r}(u):=\sum_{k=0}^{m} \binom {m}k w_{m,k,r} u^k,
\qquad
F_{m,r}^{(\ge K)}(u):=\sum_{k=K}^{m}\binom {m}k w_{m,k,r} u^k.
\]
Then
\[
B_{m,K,r} =\sum_{k=K}^{m}\binom {m}k w_{m,k,r}
\E_{\rho}\big[|\eta_2(\rho)|^k\big]
=
\E_{\rho}\left[F_{m,r}^{(\ge K)}(|\eta_2(\rho)|)\right].
\]
We will use the following elementary properties of $F_{m,r}$ and $F_{m,r}^{(\ge K)}$.

\begin{lemma}[Degree-weight generating function]
\label{lem:degree-weight-generating}
For any $r\le m\le n$ and every $0\le u\le1$,
\begin{equation}
F_{m,r}(u)
=
\sum_{s=0}^r\binom {m}s
\bigl(p+(1-p)u\bigr)^s\bigl(1-p+pu\bigr)^{m-s}.
\label{eq:degree-weight-generating}
\end{equation}
Moreover,
\begin{equation}
F_{m,r}^{(\ge K)}(u)
\le
\frac{(nu)^K}{K!}e^{nu},
\qquad
F_{m,r}(u)\le e^{nu}.
\label{eq:degree-weight-envelope}
\end{equation}
Finally, for all sufficiently large $n$ and $r\le \lceil2np+100\log n\rceil+B_*$, 
\begin{equation}
F_{m,r}(1)
=
\sum_{s=0}^r\binom {m}s
\le
\exp\bigl(Cnp\log(1/p)\bigr).
\label{eq:degree-weight-entropy}
\end{equation}
\end{lemma}

\begin{proof}
By the definition of $w_{m,k,r}$ in \eqref{e:weight-w},
\[
F_{m,r}(u)
=
\sum_{S \subset[m]} \E_{Q_m}[\one_{\{\sum_i \bY_i\le r\}}  \psi_S(\bY)^2 u^{|S|} ]
=
\E_{Q_m}\left[
\one_{\{\sum_i\bY_i\le r\}}
\prod_{i=1}^{m}(1+u\psi_i(\bY)^2)
\right].
\]
Then, direct expansion over the number of ones in $\bY$ gives
\begin{align*}
F_{m,r}(u)
=
\sum_{s=0}^r\binom {m}s
\bigl(p+(1-p)u\bigr)^s\bigl(1-p+pu\bigr)^{m-s}.
\end{align*}
Moreover, since $w_{m,k,r}\le1$ and $\binom {m}k\le n^k/k!$,
\[
F_{m,r}^{(\ge K)}(u)
\le
\sum_{k=K}^{\infty}\frac{(nu)^k}{k!}
\le
\frac{(nu)^K}{K!}e^{nu},
\]
while
$F_{m,r}(u)\le\sum_{k=0}^m\binom {m}k u^k=(1+u)^{m}\le e^{nu}$.

Finally, setting $u=1$ in the explicit formula gives
$F_{m,r}(1)=\sum_{s\le r}\binom {m}s$. Since $np\gg\log n$, for
$r\le2np+100\log n+B_*$ we have $r\le3np$ for all sufficiently large
$n$. Hence the standard binomial-sum bound yields
\[
F_{m,r}(1)
\le
(r+1)\left(\frac{en}{r}\right)^r
\le
\exp\bigl(Cnp\log(1/p)\bigr),
\]
which proves the lemma.
\end{proof}

We now complete the proof of Proposition~\ref{p:bounding-B}. 
\begin{proof}[Proof of Proposition~\ref{p:bounding-B}]
Recall that $B_{m,K,r}=\E_{\rho}\left[F_{m,r}^{(\ge K)}(|\eta_2(\rho)|)\right]$. 
We split according to
\[
\calA
:=
\left\{
|\rho|+d^{-1/2}\le \frac{c_0}{L}
\right\},
\]
where $c_0>0$ is a sufficiently small numerical constant.

On $\calA$, we have $|\eta_2(\rho)|$ is roughly of order $pL|\rho|$. More precisely, applying \cite[Theorem 2, the case $k=2$ and Lemma 16]{du2026resolution} gives
\begin{equation}\label{eq:eta_2-fine}
|\eta_2(z,z')|
\le CpL|\rho|+CpL^2(\rho^2+d^{-1}).
\end{equation}
Thus, denoting $H=npL$, for a sufficiently small fixed $\varepsilon>0$, we have a crude bound
\begin{align}
n|\eta_2(\rho)|&\le CH|\rho|+CHL\rho^2+\frac{CHL}{d}\nonumber\\
&\le \left(\varepsilon d\rho^2+\frac{4C^2H^2}{\varepsilon d}\right)+\frac{CHL}{d}\cdot d\rho^2+\frac{CHL}{d}\nonumber\\
&\le \varepsilon d\rho^2+o(1)+o(1)\cdot d\rho^2+o(1)\le 2\varepsilon d\rho^2+1\label{eq:neta_2-crude}
\end{align}
where the last two inequalities hold for large enough $n$ since $d\ge H^3$. 
Therefore, by Lemma~\ref{lem:degree-weight-generating}-\eqref{eq:degree-weight-envelope}, the two bounds \eqref{eq:eta_2-fine}, \eqref{eq:neta_2-crude}, and H\"{o}lder's inequality,
\begin{align}
\E\left[
F_{m,r}^{(\ge K)}(|\eta_2(\rho)|)\one_{\calA}
\right]
&\stackrel{\eqref{eq:degree-weight-envelope}}{\le}\frac{1}{K!}\mathbb{E}\left[(n|\eta_2(\rho)|)^Ke^{n|\eta_2(\rho)|}\one_{\calA}\right]\stackrel{\eqref{eq:neta_2-crude}}{\le} \frac{C}{K!}\mathbb{E}\left[(n|\eta_2(\rho)|)^Ke^{2\varepsilon d\rho^2}\one_{\calA}\right]\nonumber\\
&\stackrel{\eqref{eq:eta_2-fine}}{\le} CH^K\left(\mathbb{E}\left[|\rho|^Ke^{2\varepsilon d\rho^2}\right]+\E\left[L^K(\rho^2+d^{-1})^Ke^{2\varepsilon d\rho^2}\right]\right),\nonumber
\end{align}
where $C_K$ is a $K$-dependent constant. Now, we use the standard spherical overlap estimates
$\E[|\sqrt d\,\rho|^j]\le(C\sqrt j)^j$ and
$\sup_d\E [e^{c_sd\rho^2}]<\infty$ for some small numerical constant $c_s>0$. By choosing $\varepsilon<c_s/4$, and applying Cauchy-Schwartz, we obtain that
\begin{equation}\label{eq:good-event-bound}
\E\left[
F_{m,r}^{(\ge K)}(|\eta_2(\rho)|)\one_{\calA}
\right]\le C\left(\frac{H}{\sqrt{d}}\right)^K+C\left(\frac{HL}{d}\right)^K\,,
\end{equation}
with a (potentially adjusted) $K$-dependent constant $C$. 

We next consider the contribution from $\calA^c$. Since $|\eta_2|\le1$ and
$F_{m,r}^{(\ge K)}\le F_{m,r}$, Lemma~\ref{lem:degree-weight-generating}
gives
\[
\E\left[
F_{m,r}^{(\ge K)}(|\eta_2(\rho)|)\one_{\calA^c}
\right]
\le
F_{m,r}(1)\,\mathbb{P}(\calA^c).
\]
As $d^{-1/2}=o(L^{-1})$, the event $\calA^c$ implies
$|\rho|\ge c/L$. Hence the standard spherical tail bound and
$F_{m,r}(1)\le e^{CH}$ yield
\[
\E\left[
F_{m,r}^{(\ge K)}(|\eta_2(z,z')|)\one_{\calA^c}
\right]
\le
\exp\left(CH-\frac{d}{CL^2}\right).
\]
Altogether, we obtain
\begin{align*}
\E_{\rho}\left[F_{m,r}^{(\ge K)}(|\eta_2(\rho)|)\right]&=\E_{\rho}\left[F_{m,r}^{(\ge K)}(|\eta_2(\rho)|)\one_{\calA}\right]+\E_{\rho}\left[F_{m,r}^{(\ge K)}(|\eta_2(\rho)|)\one_{\calA^c}\right]\\
&\le C\left(\left(\frac{H}{\sqrt{d}}\right)^K+C\left(\frac{HL}{d}\right)^K+\exp\left(CH-\frac{d}{CL^2}\right)\right)
\end{align*}
as desired. 

Finally, since $d\ge H^3$, and $H\ge n^\delta,L=O(\log n)$, we have for sufficiently large $n$, 
\begin{equation*}
\max\left\{\frac{H}{\sqrt d},\frac{HL}{d}\right\}\le n^{-\delta/2}\,,\qquad \exp\left(CH-\frac{d}{CL^2}\right)=n^{-\omega(1)}\,.
\end{equation*}
Thus, for any fixed $K>2/\delta$, we have $B_{m,K,r}=o(1/n)$. 
\end{proof}

\section{The high-degree harmonic remainder}
\label{sec:high-harmonic-remainder}

In this section, we prove Lemma~\ref{lem:state-harmonic-expansion} 
on the harmonic expansion of $\eta_{D,S}$ and remainder estimates. 
The same argument applies to 
Theorem~\ref{thm:harmonic-expansion} for the harmonic expansion for $\eta_k$.
  Throughout this section, let $\rho=\Iprod{x}{y}$ where $x,y$ are iid uniform on $\mathbb S^{d-1}$. It is well known that 
$\rho$ is subgaussian with variance
proxy $1/d$; see e.g.\ \cite[Lemma~7]{du2026resolution}.



\subsection{Properties of the averaging operator}
\label{sec:Tproperty}
Recall the operator $\sfT_{\bfa}^{u,S}$ defined by \eqref{eq:generated-state-definition},
which is useful for averaging out the marked latent vector $x_u$ in the cavity calculation in \prettyref{sec:cavity-harmonic}.
This operator satisfies the following properties
that are important for our application.

\begin{lemma}
\label{lmm:Tproperty}
Let $D \in \calH_{\bsell}$ for $\bsell = (\ell_v: v\in V) \in \naturals^V$.
Fix a marked vertex $u\in V$ and denote $h=\ell_u$. 
Given $\bfa \in \naturals^S$, write $D'=D'(x_{V'}) \equiv \sfT_{\bfa}^{u,S} D$, 
where $V'\equiv (V\setminus u) \cup S$.
Then $D'$ satisfies the following:
\begin{enumerate}
    \item (Preservation of tensor product)
$D' \in \calH_{\bsell'}$,
where 
$\bsell' = (\ell'_i: i \in V')$,
and 
\[
\ell'_i := \begin{cases}
\ell_i    & i \in V\setminus \{u\} \\
a_i    & i \in S
\end{cases}
\]

    \item (Contraction)
    Let $\|\bfa\|_1\geq h $
    (otherwise $\sfA_{h;\bfa}$ and $D'$ are zero by
    \prettyref{lem:selection-rule}).  Define $b_m$ by
\[
 \Phi_{\bfa}(t):=\prod_{v\in S}P_{a_v}(t)=\sum_{m\geq0} b_mP_{m}(t).
\]Then,   recalling the notation \eqref{e:def.q.s},   we have
\begin{equation}
\E((D')^2)
=\frac{d^{2q}b_h}
 {N_{h,d}\prod_{v\in S}N_{a_v,d}}\E(D^2)
\leq \frac{C}{d^s}  \E(D^2)\,,
\label{eq:T-contraction}
\end{equation}
where $C$ denotes a constant depending only on $\|\bfa\|_1$ and $h$.
\end{enumerate}
Furthermore, both parts apply to $D \equiv 1$ with $V=\{u\}$ and $h=0$, in which case $V'=S$ and
$D'\in \calH_{\bfa}$.
\end{lemma}

\begin{proof}
  Recall that we suppress the ambient dimension $d$ from the subscripts when possible.   Let $V_0:=V\setminus\{u\}$ and abbreviate
\[
 \mathsf B_{V_0}(x_{V_0})
 :=\bigotimes_{v\in V_0}Y_{\ell_v}(x_v),
 \qquad
 \mathsf B_{\bfa,S}(x_S)
 :=\bigotimes_{i\in S}Y_{a_i}(x_i).
\]
By the tensor inner-product representation \eqref{eq:pure-block}, we have
\[
 D(x_V)
 =\Big\langle \mathsf C_D,
   Y_h(x_u)\otimes\mathsf B_{V_0}(x_{V_0})\Big\rangle\,.
\]
Applying the addition formula \eqref{e:addition.formula} to \prettyref{eq:harmonic-contraction} gives
\begin{align*}
&\mathsf A_{h;\mathbf a}(x_u,x_S)
=
\frac{d^q}{N_h\prod_{i\in S}N_{a_i}}
\int
\langle Y_h(x_u),Y_h(z)\rangle
\prod_{i\in S}
\langle Y_{a_i}(x_i),Y_{a_i}(z)\rangle
\,d\sigma(z) \\
&\qquad= \bigg\langle 
\underbrace{\frac{d^q}{N_h\prod_{i\in S}N_{a_i}}
   \int Y_h(z)\otimes
   \bigotimes_{i\in S}Y_{a_i}(z)\,d\sigma(z)}_{\mathsf C_{h;\bfa}},
   Y_h(x_u)\otimes\mathsf B_{\bfa,S}(x_S)
   \bigg\rangle\,.
\end{align*}
For multi-indices
$\boldsymbol\alpha\in\prod_{v\in V_0}[N_{\ell_v}]$ and
$\boldsymbol\beta\in\prod_{i\in S}[N_{a_i}]$, define
\[
 (\mathsf C')_{\boldsymbol\alpha,\boldsymbol\beta}
 :=N_h\sum_{r=1}^{N_h}
   (\mathsf C_D)_{r,\boldsymbol\alpha}
   (\mathsf C_{h;\bfa})_{r,\boldsymbol\beta}.
\]
Substituting the two representations into
\eqref{eq:generated-state-definition} and using \eqref{e:y.ytranspose} gives
\[
 D'(x_{V'})
 =\Big\langle \mathsf C',
   \mathsf B_{V_0}(x_{V_0})\otimes
   \mathsf B_{\bfa,S}(x_S)\Big\rangle.
\]
Therefore we conclude
\begin{equation}
 D'\in
 \bigg(\bigotimes_{v\in V_0}\calH_{\ell_v}\bigg)
 \otimes
 \bigg(\bigotimes_{i\in S}\calH_{a_i}\bigg)
 =\calH_{\bsell'}\,,
 \label{eq:TaD-tensor}
\end{equation}
proving the first property.

For the second property, there is nothing to prove when $\sfA_{h;\bfa}\equiv0$, so assume $\sfA_{h;\bfa}\not\equiv0$.
In order to average out the marked variable $x_u$, it is useful to express both $D$ and $\sfA_{h;\bfa}$ 
as inner products involving $Y_h(x_u)$.
Since the degree of $u$ in $D$ is $h$, further expanding \prettyref{eq:pure-block} 
we obtain (this is equivalent to \prettyref{eq:D-expandxu})
	\begin{equation}\label{e:D.rewritten}
	D(x_u,x_{V_0})
	=
	\sum_{j_u} 
	Y_{h,j_u}(x_u)\bigg(
	\sum_{\mathbf{j}_{V_0}} c_{\mathbf{j}}
	\prod_{v\in V_0}
	Y_{\ell_v,j_v}(x_v)\bigg)
	= \Big\langle Y_h(x_u), F(x_{V_0})\Big\rangle\,.
	\end{equation}
Furthermore, recalling the 
definition \prettyref{eq:harmonic-contraction}, we have
\begin{align}\nonumber
 \sfA_{h;\bfa}(x_u,x_S)
 &=d^q
 \int
 P_h(\ip{x_u}{z})
 \prod_{j\in S} P_{a_j}(\ip{x_j}{z})\,d\sigma(z) \\
 &\stackrel{\eqref{e:addition.formula}}{=} \bigg\langle Y_h(x_u), \frac{d^q}{N_h}
 \int Y_h(z)
 \prod_{j\in S}
 P_{a_j}(\ip{x_j}{z})\,d\sigma(z)
 \bigg\rangle
 \equiv 
 \Big\langle Y_h(x_u), G_{\bfa}(x_S)\Big\rangle
 \label{eq:harmonic-contraction1}
\end{align}
Using the orthonormality \eqref{e:y.ytranspose}, we have
\begin{align}
\label{e:D.second.mmt.F}
 \E_{x_u}[D(x_V)^2]
    &=\|F(x_{V_0})\|^2\,,\\
 \E_{x_u}[\sfA_{h;\bfa}(x_u,x_S)^2]
 &=\|G_{\bfa}(x_S)\|^2\,.
\label{e:A.second.mmt.G}
\end{align}
Define the covariance matrix
\[
 K_{\bfa}:=\E_{x_S}\left[
  G_{\bfa}(x_S)\otimes G_{\bfa}(x_S)
 \right].
\]
We next show that $K_{\bfa}$ is proportional to the identity. Note that
    \begin{align}\nonumber
 &\E_x\Big[
  P_a(\langle x,z\rangle)
  P_a(\langle x,z'\rangle)
 \Big]
 \stackrel{\eqref{e:addition.formula}}{=} 
 \frac{\E_x[\langle Y_a(x),Y_a(z)\rangle
 \langle Y_a(x),Y_a(z')\rangle]}{(N_a)^2}\\
 &\qquad\stackrel{\eqref{e:y.ytranspose}}{=}
 \frac{\E_x[\langle Y_a(z),Y_a(z')\rangle]}{(N_a)^2}
 \stackrel{\eqref{e:addition.formula}}{=}\frac{P_a(\langle z,z'\rangle)}{N_a}\,.
 \label{eq:zonal-kernel-contraction}
\end{align}
Similarly, for every $m\ge0$, since spherical harmonics of different degrees are orthogonal, we have
\begin{equation}
\label{eq:zonal-harmonic-selection}
    \int P_{m}(\langle z,z'\rangle)Y_{h,s}(z')\,d\sigma(z')
    \stackrel{\eqref{e:addition.formula}}{=}\frac{1}{N_m}
    \E_{z'}\Big[
    \langle Y_m(z),Y_m(z')\rangle
    Y_{h,s}(z')\Big]
    \stackrel{\eqref{e:y.ytranspose}}{=}
    \mathbf{1}\{m=h\} \frac{Y_{h,s}(z)}{N_h}\,.
\end{equation}
Now recall $\Phi_{\bfa}$ and its coefficients $b_m$ from the statement of the lemma.
For $r,s\in[N_h]$, expanding the definition of $G_{\bfa}$
and using independence of the variables in $x_S$ gives
\begin{align*}
 &\E_{x_S}\!\left[G_{\bfa,r}(x_S)G_{\bfa,s}(x_S)\right]\\
 &\qquad=
 \frac{d^{2q}}{(N_h)^2}
 \iint Y_{h,r}(z)Y_{h,s}(z')
 \prod_{v\in S}\E_{x_v}\Big[
  P_{a_v}(\langle x_v,z\rangle)
  P_{a_v}(\langle x_v,z'\rangle)
 \Big]
 \,d\sigma(z)\,d\sigma(z')\\
 &\qquad\stackrel{\eqref{eq:zonal-kernel-contraction}}{=}
 \frac{d^{2q}}{(N_h)^2\prod_{v\in S}N_{a_v}}
 \iint Y_{h,r}(z)Y_{h,s}(z')
 \Phi_{\bfa}(\langle z,z'\rangle)
 \,d\sigma(z)\,d\sigma(z') \\
 &\qquad=
 \frac{d^{2q}}{(N_h)^2\prod_{v\in S}N_{a_v}}
 \int Y_{h,r}(z)
 \sum_{m=0}^{\|\bfa\|_1} b_m
 \int Y_{h,s}(z') P_m (\langle z,z'\rangle)
 \,d\sigma(z')\,d\sigma(z)\\
 &\qquad\stackrel{\eqref{eq:zonal-harmonic-selection}}{=}
 \frac{d^{2q} b_h}{(N_h)^3\prod_{v\in S}N_{a_v}}
 \int Y_{h,r}(z) Y_{h,s}(z)\,d\sigma(z)
 \stackrel{\eqref{e:y.ytranspose}}{=}
 \frac{d^{2q} b_h}{(N_h)^3\prod_{v\in S}N_{a_v}}
 \mathbf{1}\{r=s\}
 \equiv \lambda \mathbf{1}\{r=s\}\,.
\end{align*}
This proves $K_{\bfa}=\lambda I_{\calH_h}$, as claimed.
Taking the trace gives
\[
 \lambda
 =\frac{\E_{x_S}[\|G_{\bfa}(x_S)\|^2]}{N_h}
 \stackrel{\eqref{e:A.second.mmt.G}}{=}
 \frac{\E_{x_u,x_S}[\sfA_{h;\bfa}(x_u,x_S)^2]}{N_h}
 = \frac{\E[(\mathsf{A}_{h;\bfa})^2]}{N_h}\,.
\]
Applying \prettyref{e:D.rewritten} and \prettyref{eq:harmonic-contraction1} to 
\prettyref{eq:generated-state-definition}, we obtain
	\[
	D'(x_{V_0},x_S)
	= N_h \Big\langle
	F(x_{V_0}),G_{\bfa}(x_S)
	\Big\rangle\,.
	\]
Since $S$ and $V$ are disjoint, the tuples $x_{V_0}$ and $x_S$ are independent. It follows using the above covariance calculation that
	\begin{align}\nonumber
	&\E[(D')^2]
	=(N_h)^2\E_{x_{V_0}}\left[
	\Big\langle
	F(x_{V_0}),K_{\bfa} F(x_{V_0})
	\Big\rangle\right]
	= \lambda (N_h)^2
	\E_{x_{V_0}} [\|F(x_{V_0})\|^2]\\
	&\qquad
	\stackrel{\eqref{e:D.second.mmt.F}}{=} \lambda (N_h)^2
	\E_{x_u,x_{V_0}} [D(x_u,x_{V_0})^2]
	=\frac{d^{2q}b_h}
 {N_h\prod_{v\in S}N_{a_v}}\E[D^2]\,.
 \label{eq:relative-transition-exact}
	\end{align}
This proves the identity in \eqref{eq:T-contraction}.

To prove the inequality in \eqref{eq:T-contraction}, it remains to bound $b_h$. Recall from \eqref{e:def.q.s} that $A\equiv\|\bfa\|_1$, $q=(h+A)/2$ and $s=q-h$.
The second upper bound in~\eqref{eq:direct-b-bounds} of
Lemma~\ref{lem:direct-b-bounds}, together with the dimension
estimate~\eqref{e:harmonic.dim.estimate}, gives
\[
 0\le b_h
 \le N_h\left(\frac{C(h+A)}{d}\right)^q
 \le C_{h,A}d^{h-q}
 =\frac{C_{h,A}}{d^s},
\]
where $C_{h,A}$ depends only on $h,A$.
Moreover, the same dimension estimate gives
\[
 N_h\prod_{i\in S}N_{a_i}
 \asymp_{h,A}d^{h+A}=d^{2q}.
\]
Combining these bounds with~\eqref{eq:relative-transition-exact}
proves the inequality in~\eqref{eq:T-contraction}.
\end{proof}

\subsection{Bounds for products of zonal polynomials}
\label{sec:uniform-zonal-products}
With
$N_\ell\equiv N_{\ell,d}$ and $P_\ell\equiv P_{\ell,d}$,
orthogonality and the
normalization of $P_m$ give
\begin{align}\nonumber
&\E_\rho[P_\ell(\rho)P_m(\rho)]
\stackrel{\eqref{e:addition.formula}}{=}
\indc{m=\ell}\frac{\E_{x,y}[\langle Y_\ell(x),Y_\ell(y)\rangle^2]}{(N_\ell)^2}\\
&\qquad\stackrel{\eqref{e:y.ytranspose}}{=}
\indc{m=\ell}
\frac{\E_x[
\langle Y_\ell(x),Y_\ell(x)\rangle
]}{N_\ell^2}
 =\frac{\indc{m=\ell}}{N_\ell}\,.
\label{eq:direct-zonal-orthogonality}
\end{align}
For $\nu\ge3$, the normalized zonal polynomials $P_{\ell,\nu}$ have the integral
representation (see for instance \cite[\href{https://dlmf.nist.gov/18.10.E4}{Equation~18.10.4}]{NIST:DLMF})
\begin{equation}
 P_{\ell,\nu}(t)
 \stackrel{\eqref{e:addition.formula}}{=}
 \frac{C_{\ell,\lambda}(t)}{C_{\ell,\lambda}(1)}
 =\E_U\bigg[
 \Big(t+\mathbf{i}\sqrt{1-t^2}\,U\Big)^\ell\bigg]\,,
 \qquad -1\le t\le1\,,
 \label{eq:direct-zonal-integral}
\end{equation}
where $\mathbf{i}=\sqrt{-1}$ is the imaginary unit and $U=\ip{z}{z'}$ for $z,z'$ sampled independently and uniformly at random from $\mathbb S^{\nu-2}$. Taking the modulus, we obtain
\begin{equation}
 |P_{\ell,\nu}(t)|
 \leq \E_U\Big[
 (t^2+U^2)^{\ell/2}\Big]\,.
 \label{eq:zonal-integral-ub}
\end{equation}
Abbreviate $A\equiv\|\bfa\|_1$ and recall the coefficients $b_m$ from
Lemma~\ref{lmm:Tproperty}:
\[
 \Phi_{\bfa}(t)\equiv\prod_{i\in S}P_{a_i}(t)
 =\sum_{m\ge 0} b_mP_m(t).
\]
(Note that $b_m$ implicitly depends on $\bfa$.)
By~\eqref{eq:direct-zonal-orthogonality},
\begin{equation}
 \frac{b_h}{N_h}
 =\E_\rho\Big[P_h(\rho)\Phi_{\bfa}(\rho)\Big]
 =\E_\rho\bigg[P_h(\rho)\prod_{i\in S}P_{a_i}(\rho)\bigg]\,.
 \label{eq:direct-b-definition}
\end{equation}
\begin{lemma}
\label{lem:direct-b-bounds}
There is an absolute constant $C$ such that, for every $d\ge3$,
$h\ge0$, $k\ge1$, and $\bfa\in\naturals^k$, 
\begin{equation}
 0\le\frac{b_h}{N_h}
 \le\min\bigg\{
   \frac1{N_{\max\{h,a_1,\ldots,a_k\}}},
   \bigg(\frac{C(h+A)}{d}\bigg)^{(h+A)/2}
 \bigg\}\,,
 \label{eq:direct-b-bounds}
\end{equation}
where we recall that $A=\|\mathbf a\|_1$.
Moreover, $b_h=0$ unless $A\ge h$ and $h+A$ is even.
\end{lemma}

\begin{proof}
Without loss of generality, let $S = [k]$ and $\bfa=(a_1,\ldots,a_k)\in\naturals^k$.
For notational convenience, let $a_0\equiv h$ and $a_j=\max \{a_i : i\in[0,k]\}$. Define
\[
 F_j(y)\equiv
 \bigotimes_{i\in[0,k]\setminus j}
 Y_{a_i}(y)
\]
where 
$Y_a = (Y_{a,1},\ldots,Y_{a,N_a})$ is an orthonormal basis of $\calH_a$ (cf.~\prettyref{e:y.ytranspose}). 
Let $x,y$ be iid uniform on $S^{d-1}$ and $\rho=\Iprod{x}{y}$. By~\eqref{eq:direct-b-definition} and the addition
formula~\eqref{e:addition.formula},
\[
\frac{b_h}{N_h}
= \frac{\E[\Iprod{Y_{a_j}(x) \otimes F_j(x)}{Y_{a_j}(y)\otimes F_j(y)}]}{\prod_{i=0}^kN_{a_i}} \\
=
\frac{\|\E[Y_{a_j}(y)\otimes F_j(y)]\|^2}{\prod_{i=0}^kN_{a_i}}
 \ge0\,,\]
where $\|\cdot\|$ denotes the Euclidean (Frobenius) norm of tensors.   This proves the lower bound in \eqref{eq:direct-b-bounds}.   Furthermore,
\begin{align*}
 \Big\|\E[Y_{a_j}(y)\otimes F_j(y)]\Big\|^2
 &=\sum_{s=1}^{N_{a_j}}
   \left\|\E\left[Y_{a_j,s}(y)F_j(y)\right]\right\|^2\\
 &\stepa{\le}\E\left[\|F_j(y)\|^2\right] = 
\E\left[\prod_{i \in[0,k]\setminus j} \|Y_{a_i}(y)\|^2\right]
 \stepb{=}\prod_{i \in [0,k] \setminus j}N_{a_i}.
\end{align*}
where 
(a) applies the orthonormality of $\{Y_{a_j,s}:s=1,\ldots,N_{a_j}\}$ to each entry of the tensor $F_j$;
(b) is because $\|Y_\ell(y)\|^2=N_\ell$ for all $y$, again by the addition formula \eqref{e:addition.formula}.
  Therefore,
\[
 0\le\frac{b_h}{N_h}
 \le\frac{\prod_{i\in[0,k]\setminus j}
 N_{a_i}}{\prod_{i=0}^kN_{a_i}}
 =\frac1{N_{a_j}}
 =\frac1{N_{\max\{h,a_1,\ldots,a_k\}}}\,,
\]
  where we recall from \eqref{e:harmonic.dim.estimate} that $N_{\ell}$ is nondecreasing in $\ell$. This proves the first upper bound in \prettyref{eq:direct-b-bounds}.

Next, we prove the second upper bound in \prettyref{eq:direct-b-bounds}. 
To this end, we may assume without loss of generality that $a_0=h\geq 1$; otherwise we have $h=0$, $P_0\equiv 1$ and thus can apply the same argument to $a_1,\ldots,a_k\in\naturals$. 
Let $U=\ip{z}{z'}$ for $z,z'$ sampled uniformly at random from $\mathbb{S}^{d-2}$, and let $U_0,\ldots,U_k$ be i.i.d.\ copies of $U$, independent of $\rho$. Applying \prettyref{eq:zonal-integral-ub} and H\"older's
inequality with exponents $(h+A)/a_i$
(again recalling $A=\|\bfa\|_1$) gives
\begin{align*}
 \frac{b_h}{N_h}
 & \le\E\bigg[\prod_{i=0}^k(\rho^2+U_i^2)^{a_i/2}\bigg]
 \le\prod_{i=0}^k\bigg\{
       \Big( \E[(\rho^2+U_i^2)^{(h+A)/2}]
       \Big)^{a_i/(h+A)}
       \bigg\}\\
 &=\E\Big[(\rho^2+U_0^2)^{(h+A)/2}\Big]
 \le \bigg(\frac{C(h+A)}{d}\bigg)^{(h+A)/2}\,,
\end{align*}
where the last step is because both $\rho$ and $U_0$ are subgaussian with proxy
variance at most $C/d$. This proves the second upper bound in \eqref{eq:direct-b-bounds}.

Finally, symmetry gives $b_h=0$ when $h+A$ is odd.  
The polynomial $\Phi_{\bfa}$ has degree $A$, so it is orthogonal to $P_h$ if $A<h$, in which case we again obtain $b_h=0$. This concludes the proof.
\end{proof}

\subsection{Bounding cap coefficients}
\label{sec:uniform-cap-coefficients}


Recall the centered cap kernel $\kappa$ from \eqref{e:def.kappa}, and its 
expansion \prettyref{eq:kappa-harmonic-expansion}.   It follows from \eqref{eq:direct-zonal-orthogonality} that  
\begin{equation}
\frac{\alpha_\ell}{N_\ell} = \Expect[P_\ell(\rho) \kappa(\rho)]\,.
\label{eq:alpha-ell}
\end{equation}
Define the normalized coefficients $w_\ell:=(\alpha_\ell)^2/N_\ell$, and note that they satisfy the identity
\begin{equation}
 \sum_{\ell\ge1}w_\ell=\E[\kappa(\rho)^2]=1\,.
 \label{eq:direct-cap-energies}
\end{equation}

\begin{lemma}
\label{lem:direct-cap-bound}
There is an absolute constant $C$ such that, for $d\ge3$,
$0<p\le e^{-1}$, and every integer $1\le\ell\le d$,
the normalized coefficients $w_\ell$ defined above satisfy the bound
\begin{equation}
 w_\ell\le p\frac{[C(L+\ell)]^\ell}{\ell!}\,,
 \label{eq:direct-cap-bound}
\end{equation}
  with $L$ as in \eqref{e:def.L.eps}. Moreover, for any \emph{fixed} $M$, we can bound
\begin{equation}
 (\alpha_\ell)^2=N_\ell w_\ell
 \le C_M p(dL)^\ell
 \label{eq:fixed-cap-coefficients}
\end{equation}
for all $1\le\ell\le2M$.
\end{lemma}

\begin{proof}
Recall that $p=\Prob(\rho\geq \tau)$, and let $u_d=\E[\rho\mid\rho\ge\tau]$.  
The assumption $p \le e^{-1}$ implies 
$L\ge1$. As noted above, $\rho$ is subgaussian with variance proxy $1/d$. The subgaussian bound   and the conditional Jensen's inequality give, for $s\ge 0$,
\[
 p e^{s u_d}
 \le\E[e^{s\rho}\one\{\rho\ge\tau\}]
 \le e^{s^2/(2d)}\,.
\]
Choosing $s=(2dL)^{1/2}$ gives
\begin{equation}
 0<\tau\le u_d\le \bigg( \frac{2L}{d}\bigg)^{1/2}\,.
 \label{eq:direct-cap-mean-bound}
\end{equation}
For $\ell=1$, since $P_1(t)=t$ and $N_1=d$, the coefficient definition gives
\[
 w_1
 =d\left(\E[\kappa(\rho)\rho]\right)^2
 =\frac{dp (u_d)^2}{1-p}
 \le\frac{2pL}{1-p}
 \le4pL.
\]
This proves~\eqref{eq:direct-cap-bound} for $\ell=1$, so we assume hereafter that $\ell\ge2$. 

We next relate $w_\ell$ to the value of a zonal
polynomial.  The probability density function of $\rho$ on $(-1,1)$ is
    \[
    f_d(t)
    := c_d(1-t^2)^{(d-3)/2}\,,\quad
    c_d=\frac{\Gamma(d/2)}
        {\pi^{1/2}\Gamma((d-1)/2)}\,.
    \]
Direct integration gives
\begin{equation}\label{e:formula.for.u.d}
 p u_d= \frac{c_d (1-\tau^2)^{(d-1)/2}}{d-1}\,.
\end{equation}
Recall the Gegenbauer differential equation and the derivative identity
\cite[Appendix~B.2, in particular the discussion around Equation~(B.2.6)]{dai2013approximation}:
\begin{align*}
 \frac{d}{dt}\Big[(1-t^2)^{(d-1)/2}
    (P_{\ell,d})'(t)\Big]
 &=-\ell(\ell+d-2)(1-t^2)^{(d-3)/2}P_{\ell,d}(t)
 \,,\\
 (P_{\ell,d})'(t)  
 \stackrel{\eqref{e:addition.formula}}{=} 
 \frac{(C_{\ell,\lambda})'(t)}
 {C_{\ell,\lambda}(1)}
 &  =
 2\lambda 
 \frac{C_{\ell-1,\lambda+1}(1)}
    {C_{\ell,\lambda}(1)}
 \frac{C_{\ell-1,\lambda+1}(t)}
 {C_{\ell-1,\lambda+1}(1)} 
 \stackrel{\eqref{e:addition.formula}}{=} 
   \frac{\ell(\ell+d-2)}{d-1}P_{\ell-1,d+2}(t)
 \,.
\end{align*}
  Combining gives
    \[\frac{1}{d-1}
    \frac{d}{dt}\Big[
    (1-t^2)^{(d-1)/2}
    P_{\ell-1,d+2}(t)\Big]
    = -(1-t^2)^{(d-3)/2}
    P_{\ell,d}(t)\,,
    \]
from which it follows that
 
\begin{align*}
 &\E[P_{\ell,d}(\rho)\mid\rho\ge\tau]
  =\frac{1}{p} \int_\tau^1 
 c_d (1-t^2)^{(d-3)/2}
    P_{\ell,d}(t)dt \\
&\qquad   = \frac{c_d(1-\tau^2)^{(d-1)/2}}{p(d-1)}
P_{\ell-1,d+2}(\tau)
 \stackrel{\eqref{e:formula.for.u.d}}{=}  
 u_d P_{\ell-1,d+2}(\tau)\,.
\end{align*}
Since $P_{\ell,d}$ is centered for $\ell\ge1$, 
recalling \eqref{e:def.kappa} and~\eqref{eq:alpha-ell} gives
\begin{equation}
 \frac{\alpha_\ell}{N_\ell}
 = \frac{\E(P_{\ell,d}(\rho); \rho\ge\tau)}
    {[p(1-p)]^{1/2}}
 =\bigg(\frac{p}{1-p} \bigg)^{1/2}
 \,u_d P_{\ell-1,d+2}(\tau)\,.
 \label{eq:cap-coefficient-identity}
\end{equation}
Recalling the definition $w_\ell=(\alpha_\ell)^2/N_\ell$ gives
\begin{equation}
 w_\ell
 =\frac{p}{1-p} (u_d)^2N_{\ell,d}
   P_{\ell-1,d+2}(\tau)^2\,.
 \label{eq:direct-cap-boundary}
\end{equation}
To bound the right-hand side of \eqref{eq:direct-cap-boundary}, 
we apply \prettyref{eq:zonal-integral-ub} with $\nu=d+2$ and $U$ subgaussian with variance proxy $ O(1/d)$, so
$(\E|U|^{2m})^{1/m}\le Cm/d$ for all $m\ge1$.
Then
\begin{equation}
 |P_{m,d+2}(\tau)|^2
 \overset{\prettyref{eq:zonal-integral-ub}}{\le}\E_U(\tau^2+U^2)^m
 \le\left(\tau^2+\frac{Cm}{d}\right)^m
 \overset{\prettyref{eq:direct-cap-mean-bound}}{\le}
\left(\frac{C(L+m)}d\right)^m,
\label{eq:zonal-tau}
\end{equation}
where the middle inequality follows from the triangle inequality (in $L^m$ norm).   We will apply the above with $m=\ell-1$, which is $\ge1$ since we assumed $\ell\ge2$.   We also have the dimension bound
\[
 N_{\ell,d}
 \stackrel{\eqref{e:harmonic.dim.estimate}}{\le}\binom{d+\ell-1}{\ell}
 \le\frac{(d+\ell)^\ell}{\ell!}
 \le\frac{(2d)^\ell}{\ell!}
\]
for all $1\le\ell\le d$.
Combining \eqref{eq:direct-cap-mean-bound}, 
\eqref{eq:direct-cap-boundary}, \prettyref{eq:zonal-tau}, and the dimension bound gives
\[
 w_\ell
 \le 2p\frac{2L}{d}\frac{(2d)^\ell}{\ell!}
       \left(\frac{C(L+\ell)}d\right)^{\ell-1}
 \le p\frac{[C(L+\ell)]^\ell}{\ell!}\,,
\]
which proves~\eqref{eq:direct-cap-bound}.
  Finally, for fixed $M\ge1$, and $1\le\ell\le 2M$, the bound \eqref{eq:direct-cap-bound} simplifies to
    \[
    (\alpha_\ell)^2=N_\ell w_\ell
    \le N_\ell p (CL)^\ell 
    \stackrel{\eqref{e:harmonic.dim.estimate}}{\le} C_M p (Ld)^\ell\,,
    \]
which proves \eqref{eq:fixed-cap-coefficients}.
\end{proof}


\subsection{Proofs of the harmonic expansions}
\label{sec:proof-harmonic-expansions}

\begin{proof}[\hypertarget{proof:l.roadmap-state-harmonic-expansion}{Proof of Lemma~\ref{lem:state-harmonic-expansion}}]
Recall that $D=D(x_V)$. We have a marked vertex $u\in V$, and we denote $h=\ell_u$ and $V_0=V\setminus\{u\}$.
Fix $S\subseteq V^c$, and write $k=|S|$.
All expectations are taken over $x_i$ sampled uniformly at random from $\mathbb{S}^{d-1}$.
As in \eqref{e:def.q.s}, abbreviate
$A\equiv\|\bfa\|_1$ and 
$q \equiv q(h;\bfa) = (h+A)/2$.
Recall from~\prettyref{eq:eta_D_S_expression} the orthogonal expansion
\[
 \eta_{D,S}(x_{V_0},x_S)
 =\sum_{\bfa\in\naturals^S}\gamma_{h;\bfa}
   (\mathsf T_{\bfa}^{u,S}D)(x_{V_0},x_S),
\]
where each summand is the orthogonal projection of $\eta_{D,S}(x_{V_0},x_S)$
--- as a function of $x_S$ with $x_{V_0}$ fixed --- 
onto the space $\calH_{\bfa}$ defined by \eqref{eq:tensor-harmonics}.
Because the spaces $\calH_{\bfa}$ are orthogonal for different multidegrees $\bfa$, 
the orthogonal projection of $\eta_{D,S}$ onto the direct sum of all $\calH_{\bfa}$ with total degree $A\le 2M-h$ is
\[
\sum_{\bfa\in\naturals^S: q(h;\bfa)\le M}
       \gamma_{h;\bfa}  (\mathsf T_{\bfa}^{u,S}D)(x_{V_0},x_S)
= \sum_{\bfa\in  \calA_{k,M}(h)}
       \gamma_{h;\bfa}  (\mathsf T_{\bfa}^{u,S}D)(x_{V_0},x_S)\,,
\]
where the last step applies the selection rule in Lemma~\ref{lem:selection-rule}, which shows that the projection to $\calH_{\bfa}$ is zero unless $\bfa$ lies in the set $\calA_{k,M}(h)$   defined by \eqref{e:def.calA.h}.   This is the main term in \prettyref{eq:eta-DS-expansion}, and thus the residual term 
 $\mathcal R_{D,S,M}$ satisfies, by orthogonality,
\begin{equation}
 \E[(\mathcal R_{D,S,M})^2]
 =\sum_{\bfa\in\naturals^S:q(h;\bfa)>M}
   \gamma_{h;\bfa}^{\,2}\E[(\mathsf T_{\bfa}^{u,S}D)^2]\,,
 \label{eq:direct-parseval-tail}
\end{equation}
where we also average over $x_{V_0}$ sampled according to the uniform distribution.


Assume now that $d\ge3$ and $0<p\le e^{-1}$.
For $\bfa\in\calA_{k,M}(h)$, the condition $q\le M$ gives
$a_i\le A\le2M$.  The fixed-degree estimate
\eqref{eq:fixed-cap-coefficients}
  from Lemma~\ref{lem:direct-cap-bound}   then gives
\[
 (\gamma_{h;\bfa})^2
 \stackrel{\eqref{eq:transition-coefficient-definition}}{=}
 \frac1{d^{2q}}\prod_{i\in S}
    (\alpha_{a_i})^2
 \le C_{k,M}\frac{p^kL^Ad^A}{d^{h+A}}
 =C_{k,M}\frac{p^kL^A}{d^h}\,.
\]
This proves the first bound in~\eqref{eq:transition-scales}.
The second bound in~\eqref{eq:transition-scales} is the contraction estimate~\eqref{eq:T-contraction}
of Lemma~\ref{lmm:Tproperty}; its constant depends only on $h,A$,
and hence only on $h,M$.

It remains to show the residual error \prettyref{eq:direct-parseval-tail} satisfies the bound \eqref{eq:harmonic-remainder-quantitative}, assuming $d\ge C_0L^2$. The definition
\eqref{eq:transition-coefficient-definition} can be rewritten as
    \[(\gamma_{h;\bfa})^2
 =\frac{1}{d^{2q}}
 \prod_{i\in S}(N_{a_i}w_{a_i})\,,\]
and we recall that \eqref{eq:T-contraction} gives
\[
 \E[(\mathsf T_{\bfa}^{u,S}D)^2]
 =\frac{d^{2q} b_h}{ N_h \prod_{i\in S}N_{a_i}}\,
   \,\E(D^2)\,.\]
Multiplying the above two identities gives
\[(\gamma_{h;\bfa})^2\E[(\mathsf T_{\bfa}^{u,S}D)^2]
 =\E(D^2)\bigg(\prod_{i\in S}w_{a_i}\bigg)
 \frac{b_h}{N_h}\,.\]
Thus it is equivalent to show
\begin{equation}
 \sum_{\bfa\in\naturals^S:q(h;\bfa)>M}
   \bigg(\prod_{i\in S}w_{a_i}\bigg)
   \frac{b_h}{N_h}
 \le C_{h,k,M}\frac{p^kL^{2M+2-h}}{d^{M+1}}\,.
 \label{eq:direct-weighted-tail-goal}
\end{equation}
  from which the desired bound \eqref{eq:harmonic-remainder-quantitative} would immediately follow.   To this end, we split the left-hand side of~\eqref{eq:direct-weighted-tail-goal} into two parts, according to conditions on the quantity
\[
 Q\equiv\Big\lceil4(k+M+h+1)L\Big\rceil\,.
\]
 Note that for $C_0$ large enough (depending on $h,k,M$), the assumption $d \ge C_0 L^2$ implies $Q \le d$.

\paragraph{Case I: $\max_i a_i\ge Q$.}
On the left-hand side of~\eqref{eq:direct-weighted-tail-goal}, since each summand
is nonnegative by
Lemma~\ref{lem:direct-b-bounds}, in this part we may drop the restriction $q>M$.
Recalling \prettyref{e:harmonic.dim.estimate}, the dimension $N_\ell = N_{\ell,d}$ is nondecreasing in $\ell$, and satisfies
\begin{equation}\label{e:harmonic.dim.lbd} N_\ell
  \stackrel{\eqref{e:harmonic.dim.estimate}}{\ge} \binom{d+\ell-2}{\ell}
= \prod_{i=1}^{\ell}
\frac{d-2+i}{i}   \ge \bigg(\frac{d-2}{\ell}\bigg)^\ell
\ge \bigg(\frac{d}{2\ell}\bigg)^\ell\,,
\end{equation}
where the last inequality holds for all $d\ge4$.
For $\max_i a_i\ge Q$,
the bound \eqref{eq:direct-b-bounds} from Lemma~\ref{lem:direct-b-bounds}
gives
$b_h/N_h\le1/N_Q$.
Thus the entire sum over these degrees satisfies
\begin{align} \nonumber
&\sum_{\substack{\bfa\in\naturals^S,\\\max_i a_i\ge Q}}
   \bigg(\prod_{i\in S}w_{a_i}\bigg)
   \frac{b_h}{N_h}
 \stackrel{\eqref{eq:direct-b-bounds}}{\le}
 \frac1{N_Q}\sum_{\bfa\in\naturals^S}
 \bigg( \prod_{i\in S} w_{a_i}\bigg)
 =\frac1{N_Q}
 \bigg(\sum_{\ell\ge1} w_k\bigg)^k
 \stackrel{\eqref{eq:direct-cap-energies}}{=}\frac1{N_Q}
 \\
 &\qquad\stackrel{\eqref{e:harmonic.dim.lbd}}{\le}\left(\frac{2Q}{d}\right)^Q
 \stepa{\le}
   \frac{1}{2^Q d^{Q/2}}  
 \stepb{\le}
 \frac{p^k}
 {d^{M+1+h/2}}
 \stepc{\le} 
 \frac{p^k}{L^h d^{M+1}}\,,
 \label{eq:direct-far-tail}
\end{align}
where the last three inequalities may be justified as follows: first we choose a large $C_0$ depending on $(k,M,h)$ so that
$d\ge C_0L^2$ implies $d\ge16Q^2$, which yields (a).
Then (b) follows since 
$Q\log2\ge kL$ and $Q/2\ge M+1+h/2$,
and (c) follows since $d\ge L^2$   (these bounds are very crude). Note \eqref{eq:direct-far-tail} shows that the desired bound \eqref{eq:direct-weighted-tail-goal} is satisfied by the case I contribution.


\paragraph{Case II: $\max_i a_i<Q$.}
In this case, 
$A=\|\bfa\|_1\le k(Q-1)\le C_{h,k,M}L$. 
  By the harmonic selection rule from Lemma~\ref{lem:selection-rule}, all non-zero contributions to the left-hand side of \eqref{eq:direct-weighted-tail-goal} must have
  
$A\ge h$ and $q=(h+A)/2\in\mathbb{Z}$. Let $S_A$ denote the contribution to \eqref{eq:direct-weighted-tail-goal} from any such $A$, that is,
\[
 S_A\equiv
 \sum_{\substack{\bfa\in\naturals^S,\\\|\bfa\|_1=A,\ \max_i a_i<Q}}
   \bigg(\prod_{i\in S}w_{a_i}\bigg)
   \frac{b_h}{N_h}\,.
\]
We can then bound
\begin{align}
 S_A
 &\stepa{\le} p^k\bigg(\frac{C(h+A)}{d}\bigg)^{(h+A)/2}
       \sum_{\substack{\bfa\in\naturals^S,\\\|\bfa\|_1=A,\ \max_i a_i<Q}}
       \prod_{i\in S}\frac{[C(L+a_i)]^{a_i}}{a_i!}\notag\\
 &\stepb{\le} p^k 
       \left(\frac{C(h+A)}{d}\right)^{(h+A)/2}
       [C(L+A)]^A
       \frac{k^A}{A!}
       \notag\\
 & \le 
 \frac{C_h p^k}{d^{(h+A)/2}}
 [C_k(L+A)]^A
 \frac{(h+A)^{(h+A)/2}}{A!}  
 \stepc{\le} 
 \frac{C_h p^k}{d^{(h+A)/2}}
       [C_k(L+A)]^A\notag\\
 &\stepd{\le}  
 \frac{C_{h,k,M} p^k}{d^{(h+A)/2}}
 (C_{h,k,M}L)^A  
 \le
 \frac{C_{h,k,M}p^k}{L^h}
       \left(\frac{C_{h,k,M}^2L^2}{d}\right)^{(h+A)/2},
 \label{eq:direct-total-degree-bound}
\end{align}
where
(a) applies Lemmas~\ref{lem:direct-b-bounds}
and~\ref{lem:direct-cap-bound} since $a_i<Q\le d$;
(b) uses $a_i\le A$ and the multinomial identity
    \[
    (u_1 + \ldots + u_k)^A
    = \sum_{\bfa\in\naturals^k
    :\|\bfa\|_1=A} \binom{A}{a_1,\ldots,a_k}
    \prod_{i=1}^k (u_i)^{a_i}
    \]
applied with all $u_i=1$;
(c) uses $A\ge h$ and $A!\ge(A/e)^A$, which imply
\[
 \frac{(h+A)^{(h+A)/2}}{A!}
 \le \bigg(
 \frac{h+A}{A}\bigg)^{(h+A)/2}
 \frac{e^A}{A^{(A-h)/2}}
 \le \frac{2^{h/2}(e\sqrt2)^A}{A^{(A-h)/2}}
 \le 2^{h/2}(e\sqrt2)^A\,;
\]
and (d) follows from the above observation that $A\le C_{h,k,M}L$. For $d\ge C_0 L^2$ with $C_0$ large enough, the ratio in
\eqref{eq:direct-total-degree-bound} is at most $1/2$. Summing the geometric series in $A$ with $(h+A)/2\geq M+1$ yields
\[
 \sum_{\substack{\bfa\in\naturals^S:\\\max_i a_i<Q,\ q(h;\bfa)>M}}
   \bigg(\prod_{i\in S}w_{a_i}\bigg)
   \frac{b_h}{N_h}
 \le \frac{C_{h,k,M}p^k}{L^h}
       \bigg(\frac{L^2}{d}\bigg)^{M+1}\,.
\]
  This shows that the desired bound \eqref{eq:direct-weighted-tail-goal} is also satisfied by the case II contribution, and thus concludes the proof of \eqref{eq:harmonic-remainder-quantitative}.
\end{proof}

\begin{proof}[\hypertarget{proof:thm.harmonic-expansion}{Proof of Theorem~\ref{thm:harmonic-expansion}}]
Take $V=\{u\}$, $h=0$, $D\equiv1$, and $S=[k]$, with $u\notin S$.
Then $\eta_{D,S}=\eta_k$ by~\eqref{eq:eta} and~\eqref{eq:eta-DS}.

Note that although 
Lemma~\ref{lem:state-harmonic-expansion} assumes $D \in \calH_{\bsell}$ for $\bsell\in\naturals^V$, the same proof applies also to the constant case: 
Lemma~\ref{lmm:Tproperty} continues to hold for $D\equiv 1$, and the orthogonal expansion,
coefficient and remainder estimates apply with $h=0$ and $\E(D^2)=1$.
Specializing~\eqref{e:def.q.s}
and \eqref{eq:transition-coefficient-definition} to $h=0$, we obtain
	\begin{align*}
	q & =q(0;\bfa)
	 =\frac{\|\bfa\|_1}{2}
	=s(0;\bfa)\,, \\
    \gamma_{\bfa}
	&:= \gamma_{0;\bfa}
	= \frac{1}{d^q} \prod_{j=1}^k \alpha_{a_j}\,.
	\end{align*}
Since $P_0\equiv1$ and $N_0=1$, \eqref{eq:harmonic-contraction} and \eqref{eq:generated-state-definition} become
	\begin{equation}
	\label{eq:initial-Da}
	D_{\bfa}(x_1,\ldots,x_k)
	:=(\mathsf T_{\bfa}^{u,[k]}1)(x_1,\ldots,x_k)
    =d^q\int\prod_{j=1}^k
       P_{a_j}(\langle x_j,z\rangle)\,d\sigma(z) = 
			\sfA_{0;\bfa}(x_1,\ldots,x_k).
	\end{equation}
Comparing \eqref{e:def.calA} with \eqref{e:def.calA.h} gives
	\[
 \calA_{k,M}(0)
  =\bigcup_{q=\lceil k/2\rceil}^{M}\calA_{k,q}\,.
 \]
Thus \prettyref{eq:eta-DS-expansion} specialized to $h=0$ yields
\eqref{eq:harmonic-expansion}. Moreover, since
$D_{\bfa} = \sfA_{0;\bfa} \in\calH_{\bfa}$, 
\eqref{eq:harmonic-expansion} is an orthogonal expansion, where the main term is the projection of $\eta_k$ onto the direct sum of $\calH_{\bfa}$ for total degree $\|\bfa\|_1$ up to $2M$.
For $d\ge3$ and $0<p\le e^{-1}$, 
\prettyref{eq:transition-scales} applies to $h=0$ and yields
\eqref{eq:harmonic-component-scales}.  

For $q=1$, by definition \eqref{e:def.calA}, we have 
$\mathcal{A}_{2,1}
= \{(1,1)\}$ and $\mathcal{A}_{k,1}
= \emptyset$ for $k>2$. Since $P_1(\rho)=\rho$ and $N_1=d$, the definition
\eqref{eq:initial-Da} further specializes to
	\[
 D_{(1,1)}(x_1,x_2)
 =d\int\langle x_1,z\rangle\langle x_2,z\rangle\,d\sigma(z)
 =\langle x_1,x_2\rangle\,,
\]
which verifies \eqref{eq:critical-component}.
Finally, if $d\ge C_0(K,M)L^2$, the remainder estimate \prettyref{eq:harmonic-remainder-quantitative} holds with $h=0$ and $\E(D^2)=1$, resulting in
\eqref{eq:initial-harmonic-tail}.
\end{proof}

\section*{Acknowledgment}
H.~Du and N.~Sun are supported in part by NSF-Simons collaboration grant DMS-2031883 and NSF grant DMS-2347177.
C.~Mao is supported in part by NSF CAREER Award 2338062. J.~Xu is supported in part by NSF CAREER Award CCF-2144593.

\hypertarget{sec:AI_statement}{\section*{AI methodologies statement}}
Building on the strategy developed in~\cite{du2026resolution}, which established
the detection threshold for $p\gtrsim n^{-2/3}/\log n$, OpenAI's GPT-5.6 was
used substantively in developing the proof in the present work. At an initial
stage, the model proposed an extension of the argument in~\cite{du2026resolution} to  $p\gtrsim n^{-3/4}$. A key new ingredient in this extension  was the
typical-degree truncation used in Section~\ref{sec:high-level-fourier} to control
the high-frequency Fourier remainder. 

Following this initial improvement, a proof of the detection threshold covering the full regime $p \ge n^{-1+\delta}$ emerged through an extensive interactive process between the authors and GPT-5.6 spanning many Codex sessions. The authors provided mathematical context, proposed intermediate claims, and
supplied objections and corrections, while the model proposed and revised proof
strategies and technical arguments. The model therefore materially influenced
the proof methodology underlying Sections~\ref{sec:pf-main-thm}--\ref{sec:high-harmonic-remainder}. 

During the
subsequent verification and revision, the authors 
substantially simplified the argument.  One important simplification was the authors' observation that the
cavity recurrence can be closed uniformly over an infinite class of
bounded-degree spherical harmonics, 
without first
reducing the recurrence to a finite family;  
see Section~\ref{sec:closing-recursion}. 
To make this observation precise,
the authors introduced the tensor-product spaces \eqref{eq:tensor-harmonics} of spherical harmonic polynomials, and then proved
in Lemma~\ref{lmm:Tproperty} that the cavity averaging operator \prettyref{eq:generated-state-definition} preserves this tensor-product structure  and contracts the $L_2$ norm, allowing both the recursion and the truncation to be organized in terms of 
the multidegrees, and finally closing the recurrence by taking a normalized supremum over the bounded-degree
class \eqref{eq:bounded-harmonic-class}. 
In particular, this
eliminated the need for an earlier combinatorial construction of a finite class that tracks what functions may appear in repeated application of the cavity argument and the associated counting argument.

The authors also introduced a more direct argument, intrinsic to the spherical
harmonic expansion, for controlling the high-degree harmonic remainder; see
Section~\ref{sec:high-harmonic-remainder}. This
replaced an earlier, more involved method that relies on Taylor expansion along the interpolating model introduced 
in~\cite[Appendix~A]{du2026resolution}. Finally, the authors reorganized the
logical structure of the proof and rewrote the manuscript.

\bibliographystyle{alpha}
\bibliography{ref}
\end{document}